\documentclass[a4paper,10pt]{article}

\usepackage{verbatim}
\usepackage{charter}
\usepackage[margin=3cm]{geometry}
\usepackage[algo2e,ruled,vlined]{algorithm2e}
\usepackage{algorithm}
\usepackage{enumerate}
\usepackage{booktabs}

\usepackage{here}
\usepackage{eurosym}
\usepackage{setspace}
\usepackage{ulem}
\usepackage{algpseudocode}
\usepackage[T1]{fontenc}
\usepackage{dsfont}
\usepackage{url}
\usepackage[textwidth=2.5cm]{todonotes}
\usepackage{microtype}
\usepackage[affil-sl]{authblk}
\usepackage{tikz}
\usetikzlibrary{shapes}

\usepackage{pgfplots}
\pgfplotsset{compat=1.11}
\usepackage{adjustbox}
\usepackage{subcaption}

\usepackage[breaklinks,colorlinks,citecolor=blue,linkcolor=blue,urlcolor=blue,bookmarks=false]{hyperref}
\usepackage{amsmath}
\usepackage{amssymb}
\usepackage{amsthm}
\usepackage{mathtools}
\usepackage{nicefrac}

\usepackage{amsfonts,latexsym,graphics}
\usepackage{hyperref}
\usepackage{color}
\usepackage{multirow}

\usepackage[english]{babel}
\usepackage{epsfig}
\usepackage{graphicx}
\usepackage{bm}

\usepackage[capitalise,noabbrev]{cleveref}

\usepackage{xspace}         

\theoremstyle{definition}
\newtheorem{definition}{Definition}
\newtheorem{example}[definition]{Example}

\newtheorem{remark}[definition]{Remark}

\theoremstyle{plain}
\newtheorem{proposition}[definition]{Proposition}
\newtheorem{lemma}[definition]{Lemma}
\newtheorem{theorem}[definition]{Theorem}

\newtheorem{corollary}[definition]{Corollary}
\newcommand*{\claimproofname}{Proof}

\crefname{Claim}{claim}{claims}

\newcommand{\card}[1]{|#1|}
\newcommand{\norm}[2]{\| #1 \|_{#2}}
\newcommand{\define}{\coloneqq}

\newcommand{\R}{\mathds{R}}

\def\J{{\mbox {\boldmath $J$}}}

\def\T{\mbox{\boldmath $T$}}

\def\matrix0{{\mbox {\boldmath $O$}}}

\def\j{{\mbox{\boldmath $1$}}}

\def\x{{\mbox{\boldmath $x$}}}
\def\y{{\mbox{\boldmath $y$}}}
\def\z{{\mbox{\boldmath $z$}}}

\def\vec0{\mbox{\bf 0}}

\newcommand\tran{\mkern-2mu\raise1.25ex\hbox{$\scriptscriptstyle\top$}\mkern-3.5mu}

\def\sp{\mathop{\rm sp }\nolimits}

\def\Pa{{\mathcal{P}}}
\def\Qa{{\mathcal{Q}}}

\def\diag{\mathop{\rm diag }\nolimits}

\def\1{{\bf 1}}
\def\J{{\bf J}}

	{\end{list}}

\SetKwInOut{Input}{Input}
\SetKwInOut{Output}{Output}

\def\figa{1.0/1.0/0.0,2.0/1.0/0.5111111111111111,3.0/1.0/0.4083333333333333,4.0/1.0/0.23333333333333334,5.0/1.0/0.10714285714285714,6.0/1.0/0.0380952380952381,7.0/1.0/0.008333333333333333,8.0/1.0/0.0,9.0/1.0/0.0,10.0/1.0/0.0,1.0/2.0/1.7,2.0/2.0/0.4222222222222222,3.0/2.0/0.08333333333333333,4.0/2.0/0.009523809523809525,5.0/2.0/0.0,6.0/2.0/0.0,7.0/2.0/0.0,8.0/2.0/0.0,9.0/2.0/0.0,10.0/2.0/0.0}
\def\figb{1.0/1.0/0.0,2.0/1.0/1.0444444444444445,3.0/1.0/1.625,4.0/1.0/1.6761904761904762,5.0/1.0/1.4801587301587302,6.0/1.0/1.2,7.0/1.0/0.9,8.0/1.0/0.6,9.0/1.0/0.3,10.0/1.0/0.0,1.0/2.0/0.0,2.0/2.0/3.066666666666667,3.0/2.0/2.475,4.0/2.0/1.680952380952381,5.0/2.0/1.1071428571428572,6.0/2.0/0.7142857142857143,7.0/2.0/0.44166666666666665,8.0/2.0/0.24444444444444444,9.0/2.0/0.1,10.0/2.0/0.0,1.0/3.0/9.0,2.0/3.0/2.6,3.0/3.0/0.8833333333333333,4.0/3.0/0.32857142857142857,5.0/3.0/0.11904761904761904,6.0/3.0/0.0380952380952381,7.0/3.0/0.008333333333333333,8.0/3.0/0.0,9.0/3.0/0.0,10.0/3.0/0.0}
\def\figc{1.0/1.0/0.0,2.0/1.0/1.0222222222222221,3.0/1.0/2.441666666666667,4.0/1.0/2.9,5.0/1.0/2.9285714285714284,6.0/1.0/2.7904761904761903,7.0/1.0/2.6,8.0/1.0/2.4,9.0/1.0/2.2,10.0/1.0/2.0,1.0/2.0/0.0,2.0/2.0/7.022222222222222,3.0/2.0/10.533333333333333,4.0/2.0/11.3,5.0/2.0/10.912698412698413,6.0/2.0/10.142857142857142,7.0/2.0/9.291666666666666,8.0/2.0/8.466666666666667,9.0/2.0/7.7,10.0/2.0/7.0,1.0/3.0/0.0,2.0/3.0/17.355555555555554,3.0/3.0/15.308333333333334,4.0/3.0/12.585714285714285,5.0/3.0/10.515873015873016,6.0/3.0/8.952380952380953,7.0/3.0/7.716666666666667,8.0/3.0/6.688888888888889,9.0/3.0/5.8,10.0/3.0/5.0,1.0/4.0/42.9,2.0/4.0/11.666666666666666,3.0/4.0/3.975,4.0/4.0/1.4857142857142858,5.0/4.0/0.5793650793650794,6.0/4.0/0.22857142857142856,7.0/4.0/0.08333333333333333,8.0/4.0/0.022222222222222223,9.0/4.0/0.0,10.0/4.0/0.0}
\def\figd{1.0/1.0/0.0,2.0/1.0/0.7333333333333333,3.0/1.0/1.8333333333333333,4.0/1.0/2.5476190476190474,5.0/1.0/2.857142857142857,6.0/1.0/2.961904761904762,7.0/1.0/2.9916666666666667,8.0/1.0/3.0,9.0/1.0/3.0,10.0/1.0/3.0,1.0/2.0/0.0,2.0/2.0/10.466666666666667,3.0/2.0/26.075,4.0/2.0/33.60476190476191,5.0/2.0/36.45634920634921,6.0/2.0/37.30952380952381,7.0/2.0/37.325,8.0/2.0/36.977777777777774,9.0/2.0/36.5,10.0/2.0/36.0,1.0/3.0/0.0,2.0/3.0/45.91111111111111,3.0/3.0/64.71666666666667,4.0/3.0/69.1952380952381,5.0/3.0/69.19444444444444,6.0/3.0/67.38095238095238,7.0/3.0/64.80833333333334,8.0/3.0/61.955555555555556,9.0/3.0/59.0,10.0/3.0/56.0,1.0/4.0/0.0,2.0/4.0/87.73333333333333,3.0/4.0/74.475,4.0/4.0/61.42857142857143,5.0/4.0/52.31349206349206,6.0/4.0/45.74285714285714,7.0/4.0/40.61666666666667,8.0/4.0/36.31111111111111,9.0/4.0/32.5,10.0/4.0/29.0,1.0/5.0/215.5,2.0/5.0/57.46666666666667,3.0/5.0/23.141666666666666,4.0/5.0/12.352380952380953,5.0/5.0/8.011904761904763,6.0/5.0/5.8428571428571425,7.0/5.0/4.5,8.0/5.0/3.511111111111111,9.0/5.0/2.7,10.0/5.0/2.0}
\def\fige{1.0/1.0/0.0,2.0/1.0/1.6888888888888889,3.0/1.0/5.383333333333334,4.0/1.0/6.980952380952381,5.0/1.0/7.626984126984127,6.0/1.0/7.880952380952381,7.0/1.0/7.975,8.0/1.0/8.0,9.0/1.0/8.0,10.0/1.0/8.0,1.0/2.0/0.0,2.0/2.0/13.088888888888889,3.0/2.0/36.55833333333333,4.0/2.0/50.285714285714285,5.0/2.0/56.257936507936506,6.0/2.0/58.49047619047619,7.0/2.0/59.09166666666667,8.0/2.0/59.0,9.0/2.0/58.6,10.0/2.0/58.0,1.0/3.0/0.0,2.0/3.0/82.22222222222223,3.0/3.0/174.025,4.0/3.0/218.71904761904761,5.0/3.0/237.98015873015873,6.0/3.0/245.0142857142857,7.0/3.0/246.00833333333333,8.0/3.0/244.0222222222222,9.0/3.0/240.7,10.0/3.0/237.0,1.0/4.0/0.0,2.0/4.0/271.9555555555556,3.0/4.0/364.275,4.0/4.0/379.37619047619046,5.0/4.0/375.19444444444446,6.0/4.0/365.8333333333333,7.0/4.0/355.275,8.0/4.0/344.68888888888887,9.0/4.0/334.3,10.0/4.0/324.0,1.0/5.0/0.0,2.0/5.0/424.2,3.0/5.0/329.21666666666664,4.0/5.0/265.8333333333333,5.0/5.0/231.43253968253967,6.0/5.0/211.21428571428572,7.0/5.0/197.58333333333334,8.0/5.0/187.06666666666666,9.0/5.0/178.1,10.0/5.0/170.0,1.0/6.0/1076.4,2.0/6.0/249.2888888888889,3.0/6.0/100.54166666666667,4.0/6.0/57.74761904761905,5.0/6.0/40.666666666666664,6.0/6.0/32.10952380952381,7.0/6.0/27.066666666666666,8.0/6.0/23.666666666666668,9.0/6.0/21.1,10.0/6.0/19.0}
\def\figf{1.0/1.0/0.0,2.0/1.0/0.15555555555555556,3.0/1.0/1.325,4.0/1.0/2.242857142857143,5.0/1.0/2.6984126984126986,6.0/1.0/2.895238095238095,7.0/1.0/2.975,8.0/1.0/3.0,9.0/1.0/3.0,10.0/1.0/3.0,1.0/2.0/0.0,2.0/2.0/15.577777777777778,3.0/2.0/61.0,4.0/2.0/88.14285714285714,5.0/2.0/100.5515873015873,6.0/2.0/105.97619047619048,7.0/2.0/108.45833333333333,8.0/2.0/109.73333333333333,9.0/2.0/110.5,10.0/2.0/111.0,1.0/3.0/0.0,2.0/3.0/128.66666666666666,3.0/3.0/364.19166666666666,4.0/3.0/505.8904761904762,5.0/3.0/576.1349206349206,6.0/3.0/608.9857142857143,7.0/3.0/623.35,8.0/3.0/628.5777777777778,9.0/3.0/629.3,10.0/3.0/628.0,1.0/4.0/0.0,2.0/4.0/607.9555555555555,3.0/4.0/1184.6333333333334,4.0/4.0/1444.3809523809523,5.0/4.0/1553.25,6.0/4.0/1595.695238095238,7.0/4.0/1608.4083333333333,8.0/4.0/1607.7555555555555,9.0/4.0/1601.2,10.0/4.0/1592.0,1.0/5.0/0.0,2.0/5.0/1551.2888888888888,3.0/5.0/1886.55,4.0/5.0/1891.9809523809524,5.0/5.0/1858.313492063492,6.0/5.0/1824.752380952381,7.0/5.0/1794.7,8.0/5.0/1767.2,9.0/5.0/1741.2,10.0/5.0/1716.0,1.0/6.0/0.0,2.0/6.0/2068.8,3.0/6.0/1487.1666666666667,4.0/6.0/1164.0809523809523,5.0/6.0/1003.1666666666666,6.0/6.0/917.5904761904762,7.0/6.0/866.7333333333333,8.0/6.0/832.0666666666667,9.0/6.0/805.1,10.0/6.0/782.0,1.0/7.0/5539.5,2.0/7.0/1092.7777777777778,3.0/7.0/408.10833333333335,4.0/7.0/225.86190476190475,5.0/7.0/159.76190476190476,6.0/7.0/130.8190476190476,7.0/7.0/116.33333333333333,8.0/7.0/108.15555555555555,9.0/7.0/102.9,10.0/7.0/99.0}
\def\largefiga{1.0/1.0/0.0,2.0/1.0/0.06666666666666667,3.0/1.0/0.525,4.0/1.0/0.8761904761904762,5.0/1.0/0.9880952380952381,6.0/1.0/1.0,7.0/1.0/1.0,8.0/1.0/1.0,9.0/1.0/1.0,10.0/1.0/1.0,1.0/2.0/0.0,2.0/2.0/0.7777777777777778,3.0/2.0/7.233333333333333,4.0/2.0/14.076190476190476,5.0/2.0/18.253968253968253,6.0/2.0/20.419047619047618,7.0/2.0/21.441666666666666,8.0/2.0/21.866666666666667,9.0/2.0/22.0,10.0/2.0/22.0,1.0/3.0/0.0,2.0/3.0/7.177777777777778,3.0/3.0/43.391666666666666,4.0/3.0/75.15238095238095,5.0/3.0/92.85317460317461,6.0/3.0/101.47619047619048,7.0/3.0/105.25,8.0/3.0/106.55555555555556,9.0/3.0/106.6,10.0/3.0/106.0,1.0/4.0/0.0,2.0/4.0/32.4,3.0/4.0/114.26666666666667,4.0/4.0/157.15238095238095,5.0/4.0/173.2420634920635,6.0/4.0/177.7,7.0/4.0/177.79166666666666,8.0/4.0/176.42222222222222,9.0/4.0/174.7,10.0/4.0/173.0,1.0/5.0/0.0,2.0/5.0/84.5111111111111,3.0/5.0/151.89166666666668,4.0/5.0/143.14285714285714,5.0/5.0/125.88095238095238,6.0/5.0/113.6952380952381,7.0/5.0/105.39166666666667,8.0/5.0/99.11111111111111,9.0/5.0/93.8,10.0/5.0/89.0,1.0/6.0/0.0,2.0/6.0/137.51111111111112,3.0/6.0/105.85,4.0/6.0/56.680952380952384,5.0/6.0/32.80555555555556,6.0/6.0/21.547619047619047,7.0/6.0/15.858333333333333,8.0/6.0/12.666666666666666,9.0/6.0/10.6,10.0/6.0/9.0,1.0/7.0/0.0,2.0/7.0/125.5111111111111,3.0/7.0/37.15,4.0/7.0/8.780952380952382,5.0/7.0/2.0753968253968256,6.0/7.0/0.49523809523809526,7.0/7.0/0.11666666666666667,8.0/7.0/0.022222222222222223,9.0/7.0/0.0,10.0/7.0/0.0,1.0/8.0/0.0,2.0/8.0/68.33333333333333,3.0/8.0/6.191666666666666,4.0/8.0/0.3380952380952381,5.0/8.0/0.011904761904761904,6.0/8.0/0.0,7.0/8.0/0.0,8.0/8.0/0.0,9.0/8.0/0.0,10.0/8.0/0.0,1.0/9.0/0.0,2.0/9.0/21.08888888888889,3.0/9.0/0.11666666666666667,4.0/9.0/0.0,5.0/9.0/0.0,6.0/9.0/0.0,7.0/9.0/0.0,8.0/9.0/0.0,9.0/9.0/0.0,10.0/9.0/0.0,1.0/10.0/488.5,2.0/10.0/0.0,3.0/10.0/0.0,4.0/10.0/0.0,5.0/10.0/0.0,6.0/10.0/0.0,7.0/10.0/0.0,8.0/10.0/0.0,9.0/10.0/0.0,10.0/10.0/0.0}
\def\largefigb{1.0/1.0/0.0,2.0/1.0/0.0,3.0/1.0/0.325,4.0/1.0/0.7666666666666667,5.0/1.0/0.9404761904761905,6.0/1.0/0.9904761904761905,7.0/1.0/1.0,8.0/1.0/1.0,9.0/1.0/1.0,10.0/1.0/1.0,1.0/2.0/0.0,2.0/2.0/0.24444444444444444,3.0/2.0/8.05,4.0/2.0/16.466666666666665,5.0/2.0/19.841269841269842,6.0/2.0/21.033333333333335,7.0/2.0/21.516666666666666,8.0/2.0/21.755555555555556,9.0/2.0/21.9,10.0/2.0/22.0,1.0/3.0/0.0,2.0/3.0/2.933333333333333,3.0/3.0/43.40833333333333,4.0/3.0/83.35714285714286,5.0/3.0/102.03174603174604,6.0/3.0/109.4952380952381,7.0/3.0/112.59166666666667,8.0/3.0/114.02222222222223,9.0/3.0/114.7,10.0/3.0/115.0,1.0/4.0/0.0,2.0/4.0/10.955555555555556,3.0/4.0/110.36666666666666,4.0/4.0/169.6809523809524,5.0/4.0/189.73015873015873,6.0/4.0/196.57142857142858,7.0/4.0/199.025,8.0/4.0/199.75555555555556,9.0/4.0/199.9,10.0/4.0/200.0,1.0/5.0/0.0,2.0/5.0/37.31111111111111,3.0/5.0/153.725,4.0/5.0/163.33333333333334,5.0/5.0/152.29761904761904,6.0/5.0/145.46666666666667,7.0/5.0/142.275,8.0/5.0/140.95555555555555,9.0/5.0/140.4,10.0/5.0/140.0,1.0/6.0/0.0,2.0/6.0/70.84444444444445,3.0/6.0/113.86666666666666,4.0/6.0/56.93809523809524,5.0/6.0/33.567460317460316,6.0/6.0/26.123809523809523,7.0/6.0/23.541666666666668,8.0/6.0/22.511111111111113,9.0/6.0/22.1,10.0/6.0/22.0,1.0/7.0/0.0,2.0/7.0/107.57777777777778,3.0/7.0/49.9,4.0/7.0/8.32857142857143,5.0/7.0/1.5357142857142858,6.0/7.0/0.319047619047619,7.0/7.0/0.05,8.0/7.0/0.0,9.0/7.0/0.0,10.0/7.0/0.0,1.0/8.0/0.0,2.0/8.0/111.22222222222223,3.0/8.0/15.741666666666667,4.0/8.0/1.0238095238095237,5.0/8.0/0.05555555555555555,6.0/8.0/0.0,7.0/8.0/0.0,8.0/8.0/0.0,9.0/8.0/0.0,10.0/8.0/0.0,1.0/9.0/0.0,2.0/9.0/85.08888888888889,3.0/9.0/3.7083333333333335,4.0/9.0/0.1,5.0/9.0/0.0,6.0/9.0/0.0,7.0/9.0/0.0,8.0/9.0/0.0,9.0/9.0/0.0,10.0/9.0/0.0,1.0/10.0/0.0,2.0/10.0/47.28888888888889,3.0/10.0/0.8083333333333333,4.0/10.0/0.004761904761904762,5.0/10.0/0.0,6.0/10.0/0.0,7.0/10.0/0.0,8.0/10.0/0.0,9.0/10.0/0.0,10.0/10.0/0.0,1.0/11.0/0.0,2.0/11.0/19.755555555555556,3.0/11.0/0.1,4.0/11.0/0.0,5.0/11.0/0.0,6.0/11.0/0.0,7.0/11.0/0.0,8.0/11.0/0.0,9.0/11.0/0.0,10.0/11.0/0.0,1.0/12.0/0.0,2.0/12.0/5.7555555555555555,3.0/12.0/0.0,4.0/12.0/0.0,5.0/12.0/0.0,6.0/12.0/0.0,7.0/12.0/0.0,8.0/12.0/0.0,9.0/12.0/0.0,10.0/12.0/0.0,1.0/13.0/0.0,2.0/13.0/0.8444444444444444,3.0/13.0/0.0,4.0/13.0/0.0,5.0/13.0/0.0,6.0/13.0/0.0,7.0/13.0/0.0,8.0/13.0/0.0,9.0/13.0/0.0,10.0/13.0/0.0,1.0/14.0/0.0,2.0/14.0/0.17777777777777778,3.0/14.0/0.0,4.0/14.0/0.0,5.0/14.0/0.0,6.0/14.0/0.0,7.0/14.0/0.0,8.0/14.0/0.0,9.0/14.0/0.0,10.0/14.0/0.0,1.0/15.0/500.0,2.0/15.0/0.0,3.0/15.0/0.0,4.0/15.0/0.0,5.0/15.0/0.0,6.0/15.0/0.0,7.0/15.0/0.0,8.0/15.0/0.0,9.0/15.0/0.0,10.0/15.0/0.0}
\def\largefigc{1.0/1.0/0.0,2.0/1.0/0.0,3.0/1.0/0.14166666666666666,4.0/1.0/0.3904761904761905,5.0/1.0/0.626984126984127,6.0/1.0/0.8047619047619048,7.0/1.0/0.9166666666666666,8.0/1.0/0.9777777777777777,9.0/1.0/1.0,10.0/1.0/1.0,1.0/2.0/0.0,2.0/2.0/0.044444444444444446,3.0/2.0/5.525,4.0/2.0/15.123809523809523,5.0/2.0/19.80952380952381,6.0/2.0/21.457142857142856,7.0/2.0/21.916666666666668,8.0/2.0/22.0,9.0/2.0/22.0,10.0/2.0/22.0,1.0/3.0/0.0,2.0/3.0/0.4,3.0/3.0/34.30833333333333,4.0/3.0/78.60952380952381,5.0/3.0/99.14682539682539,6.0/3.0/107.4047619047619,7.0/3.0/110.88333333333334,8.0/3.0/112.5111111111111,9.0/3.0/113.4,10.0/3.0/114.0,1.0/4.0/0.0,2.0/4.0/2.6444444444444444,3.0/4.0/93.89166666666667,4.0/4.0/168.66190476190476,5.0/4.0/192.88095238095238,6.0/4.0/200.27619047619046,7.0/4.0/202.81666666666666,8.0/4.0/203.8,9.0/4.0/204.1,10.0/4.0/204.0,1.0/5.0/0.0,2.0/5.0/10.244444444444444,3.0/5.0/141.25833333333333,4.0/5.0/158.1904761904762,5.0/5.0/146.77777777777777,6.0/5.0/140.35238095238094,7.0/5.0/137.675,8.0/5.0/136.53333333333333,9.0/5.0/136.1,10.0/5.0/136.0,1.0/6.0/0.0,2.0/6.0/28.533333333333335,3.0/6.0/119.55,4.0/6.0/64.48571428571428,5.0/6.0/37.91269841269841,6.0/6.0/28.98095238095238,7.0/6.0/25.591666666666665,8.0/6.0/24.133333333333333,9.0/6.0/23.4,10.0/6.0/23.0,1.0/7.0/0.0,2.0/7.0/55.24444444444445,3.0/7.0/67.125,4.0/7.0/12.709523809523809,5.0/7.0/2.757936507936508,6.0/7.0/0.719047619047619,7.0/7.0/0.2,8.0/7.0/0.044444444444444446,9.0/7.0/0.0,10.0/7.0/0.0,1.0/8.0/0.0,2.0/8.0/81.46666666666667,3.0/8.0/26.683333333333334,4.0/8.0/1.6095238095238096,5.0/8.0/0.08333333333333333,6.0/8.0/0.004761904761904762,7.0/8.0/0.0,8.0/8.0/0.0,9.0/8.0/0.0,10.0/8.0/0.0,1.0/9.0/0.0,2.0/9.0/94.64444444444445,3.0/9.0/8.508333333333333,4.0/9.0/0.19523809523809524,5.0/9.0/0.003968253968253968,6.0/9.0/0.0,7.0/9.0/0.0,8.0/9.0/0.0,9.0/9.0/0.0,10.0/9.0/0.0,1.0/10.0/0.0,2.0/10.0/89.68888888888888,3.0/10.0/2.275,4.0/10.0/0.023809523809523808,5.0/10.0/0.0,6.0/10.0/0.0,7.0/10.0/0.0,8.0/10.0/0.0,9.0/10.0/0.0,10.0/10.0/0.0,1.0/11.0/0.0,2.0/11.0/67.73333333333333,3.0/11.0/0.55,4.0/11.0/0.0,5.0/11.0/0.0,6.0/11.0/0.0,7.0/11.0/0.0,8.0/11.0/0.0,9.0/11.0/0.0,10.0/11.0/0.0,1.0/12.0/0.0,2.0/12.0/39.48888888888889,3.0/12.0/0.15833333333333333,4.0/12.0/0.0,5.0/12.0/0.0,6.0/12.0/0.0,7.0/12.0/0.0,8.0/12.0/0.0,9.0/12.0/0.0,10.0/12.0/0.0,1.0/13.0/0.0,2.0/13.0/19.977777777777778,3.0/13.0/0.025,4.0/13.0/0.0,5.0/13.0/0.0,6.0/13.0/0.0,7.0/13.0/0.0,8.0/13.0/0.0,9.0/13.0/0.0,10.0/13.0/0.0,1.0/14.0/0.0,2.0/14.0/7.044444444444444,3.0/14.0/0.0,4.0/14.0/0.0,5.0/14.0/0.0,6.0/14.0/0.0,7.0/14.0/0.0,8.0/14.0/0.0,9.0/14.0/0.0,10.0/14.0/0.0,1.0/15.0/0.0,2.0/15.0/2.2,3.0/15.0/0.0,4.0/15.0/0.0,5.0/15.0/0.0,6.0/15.0/0.0,7.0/15.0/0.0,8.0/15.0/0.0,9.0/15.0/0.0,10.0/15.0/0.0,1.0/16.0/0.0,2.0/16.0/0.6,3.0/16.0/0.0,4.0/16.0/0.0,5.0/16.0/0.0,6.0/16.0/0.0,7.0/16.0/0.0,8.0/16.0/0.0,9.0/16.0/0.0,10.0/16.0/0.0,1.0/17.0/0.0,2.0/17.0/0.022222222222222223,3.0/17.0/0.0,4.0/17.0/0.0,5.0/17.0/0.0,6.0/17.0/0.0,7.0/17.0/0.0,8.0/17.0/0.0,9.0/17.0/0.0,10.0/17.0/0.0,1.0/18.0/0.0,2.0/18.0/0.022222222222222223,3.0/18.0/0.0,4.0/18.0/0.0,5.0/18.0/0.0,6.0/18.0/0.0,7.0/18.0/0.0,8.0/18.0/0.0,9.0/18.0/0.0,10.0/18.0/0.0,1.0/19.0/0.0,2.0/19.0/0.0,3.0/19.0/0.0,4.0/19.0/0.0,5.0/19.0/0.0,6.0/19.0/0.0,7.0/19.0/0.0,8.0/19.0/0.0,9.0/19.0/0.0,10.0/19.0/0.0,1.0/20.0/500.0,2.0/20.0/0.0,3.0/20.0/0.0,4.0/20.0/0.0,5.0/20.0/0.0,6.0/20.0/0.0,7.0/20.0/0.0,8.0/20.0/0.0,9.0/20.0/0.0,10.0/20.0/0.0}
\def\largefigd{1.0/1.0/0.0,2.0/1.0/0.0,3.0/1.0/0.5083333333333333,4.0/1.0/0.8761904761904762,5.0/1.0/0.9722222222222222,6.0/1.0/0.9952380952380953,7.0/1.0/1.0,8.0/1.0/1.0,9.0/1.0/1.0,10.0/1.0/1.0,1.0/2.0/0.0,2.0/2.0/0.06666666666666667,3.0/2.0/7.491666666666666,4.0/2.0/16.3,5.0/2.0/20.03174603174603,6.0/2.0/21.34285714285714,7.0/2.0/21.8,8.0/2.0/21.955555555555556,9.0/2.0/22.0,10.0/2.0/22.0,1.0/3.0/0.0,2.0/3.0/0.7333333333333333,3.0/3.0/42.041666666666664,4.0/3.0/84.15714285714286,5.0/3.0/102.42063492063492,6.0/3.0/109.73809523809524,7.0/3.0/112.825,8.0/3.0/114.2,9.0/3.0/114.8,10.0/3.0/115.0,1.0/4.0/0.0,2.0/4.0/3.422222222222222,3.0/4.0/109.35,4.0/4.0/172.03809523809525,5.0/4.0/190.37301587301587,6.0/4.0/196.76666666666668,7.0/4.0/199.375,8.0/4.0/200.57777777777778,9.0/4.0/201.3,10.0/4.0/202.0,1.0/5.0/0.0,2.0/5.0/12.177777777777777,3.0/5.0/153.76666666666668,4.0/5.0/160.94761904761904,5.0/5.0/150.5079365079365,6.0/5.0/144.40952380952382,7.0/5.0/141.25833333333333,8.0/5.0/139.62222222222223,9.0/5.0/138.7,10.0/5.0/138.0,1.0/6.0/0.0,2.0/6.0/30.2,3.0/6.0/115.725,4.0/6.0/57.576190476190476,5.0/6.0/34.48809523809524,6.0/6.0/26.580952380952382,7.0/6.0/23.733333333333334,8.0/6.0/22.644444444444446,9.0/6.0/22.2,10.0/6.0/22.0,1.0/7.0/0.0,2.0/7.0/53.577777777777776,3.0/7.0/51.28333333333333,4.0/7.0/7.5095238095238095,5.0/7.0/1.1984126984126984,6.0/7.0/0.16666666666666666,7.0/7.0/0.008333333333333333,8.0/7.0/0.0,9.0/7.0/0.0,10.0/7.0/0.0,1.0/8.0/0.0,2.0/8.0/80.46666666666667,3.0/8.0/15.3,4.0/8.0/0.5476190476190477,5.0/8.0/0.007936507936507936,6.0/8.0/0.0,7.0/8.0/0.0,8.0/8.0/0.0,9.0/8.0/0.0,10.0/8.0/0.0,1.0/9.0/0.0,2.0/9.0/92.17777777777778,3.0/9.0/3.7583333333333333,4.0/9.0/0.04285714285714286,5.0/9.0/0.0,6.0/9.0/0.0,7.0/9.0/0.0,8.0/9.0/0.0,9.0/9.0/0.0,10.0/9.0/0.0,1.0/10.0/0.0,2.0/10.0/83.84444444444445,3.0/10.0/0.6166666666666667,4.0/10.0/0.004761904761904762,5.0/10.0/0.0,6.0/10.0/0.0,7.0/10.0/0.0,8.0/10.0/0.0,9.0/10.0/0.0,10.0/10.0/0.0,1.0/11.0/0.0,2.0/11.0/64.5111111111111,3.0/11.0/0.125,4.0/11.0/0.0,5.0/11.0/0.0,6.0/11.0/0.0,7.0/11.0/0.0,8.0/11.0/0.0,9.0/11.0/0.0,10.0/11.0/0.0,1.0/12.0/0.0,2.0/12.0/41.111111111111114,3.0/12.0/0.03333333333333333,4.0/12.0/0.0,5.0/12.0/0.0,6.0/12.0/0.0,7.0/12.0/0.0,8.0/12.0/0.0,9.0/12.0/0.0,10.0/12.0/0.0,1.0/13.0/0.0,2.0/13.0/21.711111111111112,3.0/13.0/0.0,4.0/13.0/0.0,5.0/13.0/0.0,6.0/13.0/0.0,7.0/13.0/0.0,8.0/13.0/0.0,9.0/13.0/0.0,10.0/13.0/0.0,1.0/14.0/0.0,2.0/14.0/10.088888888888889,3.0/14.0/0.0,4.0/14.0/0.0,5.0/14.0/0.0,6.0/14.0/0.0,7.0/14.0/0.0,8.0/14.0/0.0,9.0/14.0/0.0,10.0/14.0/0.0,1.0/15.0/0.0,2.0/15.0/4.2,3.0/15.0/0.0,4.0/15.0/0.0,5.0/15.0/0.0,6.0/15.0/0.0,7.0/15.0/0.0,8.0/15.0/0.0,9.0/15.0/0.0,10.0/15.0/0.0,1.0/16.0/0.0,2.0/16.0/1.3111111111111111,3.0/16.0/0.0,4.0/16.0/0.0,5.0/16.0/0.0,6.0/16.0/0.0,7.0/16.0/0.0,8.0/16.0/0.0,9.0/16.0/0.0,10.0/16.0/0.0,1.0/17.0/0.0,2.0/17.0/0.35555555555555557,3.0/17.0/0.0,4.0/17.0/0.0,5.0/17.0/0.0,6.0/17.0/0.0,7.0/17.0/0.0,8.0/17.0/0.0,9.0/17.0/0.0,10.0/17.0/0.0,1.0/18.0/0.0,2.0/18.0/0.044444444444444446,3.0/18.0/0.0,4.0/18.0/0.0,5.0/18.0/0.0,6.0/18.0/0.0,7.0/18.0/0.0,8.0/18.0/0.0,9.0/18.0/0.0,10.0/18.0/0.0,1.0/19.0/0.0,2.0/19.0/0.0,3.0/19.0/0.0,4.0/19.0/0.0,5.0/19.0/0.0,6.0/19.0/0.0,7.0/19.0/0.0,8.0/19.0/0.0,9.0/19.0/0.0,10.0/19.0/0.0,1.0/20.0/0.0,2.0/20.0/0.0,3.0/20.0/0.0,4.0/20.0/0.0,5.0/20.0/0.0,6.0/20.0/0.0,7.0/20.0/0.0,8.0/20.0/0.0,9.0/20.0/0.0,10.0/20.0/0.0,1.0/21.0/0.0,2.0/21.0/0.0,3.0/21.0/0.0,4.0/21.0/0.0,5.0/21.0/0.0,6.0/21.0/0.0,7.0/21.0/0.0,8.0/21.0/0.0,9.0/21.0/0.0,10.0/21.0/0.0,1.0/22.0/0.0,2.0/22.0/0.0,3.0/22.0/0.0,4.0/22.0/0.0,5.0/22.0/0.0,6.0/22.0/0.0,7.0/22.0/0.0,8.0/22.0/0.0,9.0/22.0/0.0,10.0/22.0/0.0,1.0/23.0/0.0,2.0/23.0/0.0,3.0/23.0/0.0,4.0/23.0/0.0,5.0/23.0/0.0,6.0/23.0/0.0,7.0/23.0/0.0,8.0/23.0/0.0,9.0/23.0/0.0,10.0/23.0/0.0,1.0/24.0/0.0,2.0/24.0/0.0,3.0/24.0/0.0,4.0/24.0/0.0,5.0/24.0/0.0,6.0/24.0/0.0,7.0/24.0/0.0,8.0/24.0/0.0,9.0/24.0/0.0,10.0/24.0/0.0,1.0/25.0/500.0,2.0/25.0/0.0,3.0/25.0/0.0,4.0/25.0/0.0,5.0/25.0/0.0,6.0/25.0/0.0,7.0/25.0/0.0,8.0/25.0/0.0,9.0/25.0/0.0,10.0/25.0/0.0}

\newread\myread
\newcommand{\makeplot}[4]{ 
    \begin{tikzpicture}
        \begin{axis}[grid=none,ymin=0,ymax=#3,xmax=#2,xmin=0,xtick=data,ytick=data,minor tick num=1,axis lines = middle,xlabel=number of partitions joined,ylabel=number of cells,x label style={at={(axis description cs:0.5,-0.1)},anchor=north},y label style={at={(axis description cs:-0.1,.5)},rotate=90,anchor=south}]

            \foreach \x/\y/\z in #1 {
                \newdimen\dummyDim 
                \dummyDim = \x pt
                \ifdim \dummyDim > 1pt 
                    \edef\temp{\noexpand\node[fill=black,shape=circle,inner sep=0pt,minimum size=\z*#4 mm] at (\x,\y) {};}
                    \temp
                \else
                \fi
            }
        \end{axis}
    \end{tikzpicture}
}

\title{Characterizing and computing weight-equitable partitions of graphs}

\author[1,2,3]{Aida Abiad}
\author[1]{Christopher Hojny}
\author[1]{Sjanne Zeijlemaker}

\affil[1]{
	Eindhoven University of Technology
Eindhoven, The~Netherlands\newline
	\textit{email:} \{a.abiad.monge, c.hojny, s.zeijlemaker\}@tue.nl
}
\affil[2]{
	Ghent University\\
	Department of Mathematics: Analysis, Logic and Discrete Mathematics\\
	Belgium
}
\affil[3]{
	Vrije Universiteit Brussel\\
	Department of Mathematics and Data Science\\
	 Belgium
}

\date{}

\begin{document}
	
	\maketitle
	
	\begin{abstract}
Weight-equitable partitions of graphs, which are a natural extension of the well-known equitable partitions, have been shown to be a powerful tool to weaken the regularity assumption in several well-known eigenvalue bounds. In this work we aim to further our algebraic and computational understanding of weight-equitable partitions. We do so by showing several spectral properties and algebraic characterizations, and by providing a method to find coarse weight-equitable partitions.
	\end{abstract}
	
	\normalem
	
	\section{Introduction}

	Let~$G = (V,E)$ be a simple undirected connected graph with $n$ vertices, and let~$A$ denote its
	adjacency matrix.
	Many properties of~$G$ such as regularity or bipartiteness can be characterized from the spectrum of~$A$.
	If~$G$ is large, however, investigating the spectrum of~$G$ might be
	cumbersome, which motivates to study ``condensed'' versions of~$A$ that
	preserve properties of its spectrum.
	
	One of the most popular methods to shrink~$A$ is based on equitable
	partitions.
	To define this, let~$\mathcal{P} = \{V_1, \dots, V_m\}$ ($m<n$) be a partition
	of~$V$, and, for~$i, j \in [m] \define \{1,\dots,m\}$ and~$u \in V_i$,
	let~$b_{ij}(u)$ be the number of neighbors of~$u$ in~$V_j$.
	The partition~$\mathcal{P}$ is called \emph{equitable} (or \emph{regular}) if~$b_{ij}(u)$ is
	independent from the concrete choice of~$u \in V_i$, i.e., $b_{ij}(u) =
	b_{ij}(v)$ for all~$u, v \in V_i$.
        In this case, the matrix $B=(b_{ij})_{i,j \in [m]}$ is called the \emph{quotient matrix} of $\mathcal{P}$.
        Since it is known that for an
        equitable partition the eigenvalues of $B$ are also eigenvalues of $A$, see Godsil and Royle \cite{GR2001}, some spectral properties of~$B$ carry over to~$A$.
	Equitable partitions have been proven to be useful to derive, among
        many others, sharp eigenvalue bounds on the independence number
        like the celebrated ratio bound by Hoffman~\cite{H1970}, but such results only hold when the underlying graph is regular.
	
	To be able to derive graph properties from the spectrum also in the broader
	context of general graphs, a natural generalization of equitable partitions is to assign each
	vertex of~$G$ a weight such that~$G$ is ``weight regularized'', which leads
	to the concept of weight-equitable partitions. Weight-equitable partitions have been shown to be a powerful tool to extend several classical results for non-regular graphs.
	They were first used by Haemers in~1979 \cite[Theorem~6]{BS1979} to provide
	a proof of Hoffman's lower bound for the chromatic
	number of a general graph, weakening the regularity assumption required in the well-known Hoffman result on the independence number \cite{H1970}.
	Such a bound for general graphs has recently been extended to the distance
	$k$-chromatic number also using weight-equitable partitions, see~\cite[Theorem~4.3]{acfns2020}.
	Fiol and Garriga \cite{FG1999, F1999} used them to obtain several sharp spectral
	bounds for parameters of non-regular graphs.
	Examples of such results are an extension of Hoffman's ratio bound for the
	chromatic number or a generalization of the Lov\'asz bound for the Shannon
	capacity of a graph.
	Moreover, Fiol \cite{F1999} used weight-equitable partitions to show that a bound
	for the weight-independence number is best possible, and Fiol and Garriga  \cite{FG1999} used them to
	obtain spectral characterizations of distance-regularity around a set and
	spectral characterizations of completely regular codes.
	Recently, new algebraic characterizations of
	weight-equitable partitions and a new application of such partitions to improve the
	classical Hoffman's ratio bound were shown in \cite{A2019}.

Note that there is a trade-off between the two main goals
	of (weight-) equitable partitions.
	On the one hand, the coarser the partition~$\mathcal{P}$ is (i.e., the smaller~$m$), the smaller the subset of the spectrum of~$A$
	that can be recovered from the spectrum of~$B$.
	On the other hand,  the finer the partition ~$\mathcal{P}$ is (i.e., the larger~$m$), the more information on the spectrum of $A$ that can be recovered from
	the spectrum of $B$. Depending on the aimed result, one might be interested in either of the two
	extremes. If one is mostly concerned about shrinking~$A$ as much as
	possible, one is interested in finding a coarsest (weight-) equitable
	partition. For instance, for bounding the independence number of a graph, one only needs a weight partition with $m=2$ cells (see Fiol \cite{F1999}), but for characterizing pseudo distance-regular graphs one needs to weight partition the graph into $m=d+1$ cells, where $d+1$ is the number of distinct eigenvalues (see Fiol \cite{F2001}).
Another application of the coarsest equitable partition in linear programming was shown by Grohe et al.~\cite{GKMS2014}.
While for the equitable case,  Bastert~\cite{B1999} showed that a coarsest equitable partition can be found
	very efficiently,  we are not aware of any result
	in this direction for weight-equitable partitions.

	The aim of this article is thus to better understand weight-equitable partitions from the algebraic and computational point of view,
	and to develop means to find coarse weight-equitable partitions.
	To this end, we derive novel algebraic characterizations of weight-equitable
	partitions.
        Several known characterizations of equitable partitions follow as a corollary of our results.
	Moreover, we devise an operator that turns fine weight-equitable partitions
	into coarser ones.
	Since devising an algorithm to find a coarsest non-trivial weight-equitable
	partition is open, we investigate the potential of this operator in
	producing coarse partitions via computational experiments.
	Our computational results show that this operator is able to produce very coarse
	partitions in many cases, allowing to achieve a significant reduction of
	the size of~$A$.

	The outline of this article is as follows.
        Section~\ref{sec:weightquotient} introduces our notation as well as
        basic definitions.
        In Section~\ref{sec:properties}, we derive spectral
        properties of weight-equitable partitions, whereas
        Section~\ref{sec:characterizations} provides novel characterizations
        of weight-equitable partitions and operators to generate them.
        Section~\ref{sec:algo} investigates the potential of one such
        operator to produce coarse weight-equitable partitions.

	\section{Basic Definitions and Notation}\label{sec:weightquotient}
	
	Throughout this article, we denote by~$\J$ an all-ones matrix and
        by~$\j$ an all-ones vector whose dimensions will be clear from the
        context.
        The~$i$-th canonical vector (of suitable dimension) is denoted
        by~$e_i$, and~$\norm{\cdot}{}$ denotes the Euclidean norm of a vector.
        Moreover, for a finite set~$V$, we denote its powerset by~$\mathfrak{P}(V)$.
        For a simple undirected connected graph~$G = (V, E)$, we denote by~$n = |V|$
        the number of its vertices and by~$A = A(G)$ its adjacency matrix.
        Throughout this article, we always assume~$G$ to be simple, undirected, and connected if not stated differently.
        Moreover, we assume the vertices to be labeled~$1, \dots, n$, i.e., $V = [n]$.
        The set~$G(u)$ denotes the neighborhood of a vertex~$u \in V$, i.e.\ the set of vertices adjacent to $u$, and
        we write~$u \sim v$ if~$u,v \in V$ are adjacent.
        The automorphism group of~$G$ is denoted by~$\text{Aut}(G)$.

        The eigenvalues of~$A$ are given by~$\lambda_1, \dots, \lambda_n$,
        and we assume from now on that the eigenvalues are sorted
        non-increasingly, i.e., $\lambda_1 \geq \dots \geq \lambda_n$.
        We denote the \emph{spectrum} of~$G$ by
        \[
          \sp(G) \define \sp(A) \define
          \{\theta_0^{m_0}, \theta_1^{m_1}, \dots, \theta_d^{m_d}\},
        \]
        where~$\theta_0 > \theta_1 > \dots > \theta_d$ are the distinct
        eigenvalues of~$A$ in decreasing order with multiplicities~$m_i =
        m(\theta_i)$, $i \in \{0\} \cup [d]$.
        Note that $\theta_0=\lambda_1$ and $\theta_d=\lambda_n$.
	Since $G$ is connected (so $A$ is irreducible), the Perron-Frobenius
        Theorem assures that $\lambda_{1}$ is simple, positive, and has a
        positive eigenvector.
        If $G$ is disconnected, the existence of such an eigenvector is not
        guaranteed, unless all its connected components have the same
        maximum eigenvalue.
        Throughout this work, the positive eigenvector associated with the largest
	(positive and with multiplicity one) eigenvalue $\lambda_{1}$ is
	denoted by $\nu=(\nu_{1},\ldots,\nu_{n})^{\top}$.
        This eigenvector is called the \emph{Perron eigenvector}, and
        we assume it to be normalized such that its minimum
        entry is $1$.
	For instance, if $G$ is regular, we have $\nu=\j$.
		
        To be able to define weight-equitable partitions of a connected
        simple graph~$G$ with Perron eigenvector~$\nu$, we consider the
        map~$\rho\colon \mathfrak{P}(V) \to \R^n$, defined by $\rho(U) \define
        \sum_{u\in U} \nu_{u} e_{u}$ for any $U\neq \emptyset$.
            By convention, $\rho(\emptyset) = 0$, and we write~$\rho(u)$
        instead of~$\rho(\{u\})$.
        Since~$\rho$ is linear, we can interpret it to assign each~$u \in
        V$ the weight~$\rho(u) = \nu_u$.
        Doing so, we ``regularize'' the graph, in the sense that the
        \emph{weight-degree} $\delta_u^{\ast}$ of each vertex $u\in V$
        becomes a constant, where
        \[
          \delta^{*}_{u}
          \define
          \frac{1}{\nu_{u}}\sum_{v\in G (u)}\nu_{v}
          =
          \lambda_{1}.
        \]	
	If $\mathcal{P}$ is a partition of the vertex set $V=V_{1}\cup
	\cdots \cup V_{m}$, the \emph{weight-intersection number}
        of $u\in V_{i}$, $i \in [m]$, is
        \begin{align*}
          b^{*}_{ij}(u)
          &\define
          \frac{1}{\nu_{u}}\sum_{v \in G(u)\cap V_{j}}\nu_{v},
          && i,j \in [m].
        \end{align*}
	Observe that the sum of the weight-intersection numbers for all
	$j \in [m]$ gives the weight-degree of each vertex $u\in
	V_{i}$:
	\[
          \sum_{j=1}^{m}b^{*}_{ij}(u)=\frac{1}{\nu_{u}}\sum_{v \in
            G(u)}\nu_{v}=\delta^{*}_{u}=\lambda_{1}.
        \]
        Using these definitions, we are now able to define weight-equitable
        partitions.
        \begin{definition}
          Let~$G$ be a connected simple graph and let~$\mathcal{P}
          = (V_1, \dots, V_m)$ be a partition of~$V$.
          Then, $\mathcal{P}$ is called \emph{weight-equitable} (or
          \emph{weight-regular)} if~$b^*_{ij}(u) = b^*_{ij}(v)$ for
          all~$i,j \in [m]$ and~$u,v \in V_i$.
          That is, the weight-intersection numbers do not depend on the
          vertex~$u \in V_i$.
          In this case, we write~$b^*_{ij}$ instead of~$b^{*}_{ij}(u)$, $u
          \in V_i$.
        \end{definition}

	A matrix characterization of weight-equitable partitions can be done via the following matrix associated with any partition $\mathcal{P}$.
    The \emph{weight-characteristic matrix} of $\mathcal{P}$ is the $n\times m$ matrix $\tilde{S}^{*}=(\tilde{s}^{*}_{uj})$ with entries
	\[
	\tilde{s}^{*}_{uj}
	=
	\begin{cases}
		\nu_u, & \text{if } u \in V_j,\\
		0, & \text{otherwise},
	\end{cases}
	\]
	for all $(u,j)\in V\times [m]$ and, hence, satisfying $(\tilde{S}^{*})^{\top}\tilde{S}^{*}=D^{2}$, where $D=\diag (\|\rho(V_{1})\|,\ldots,\|\rho (V_{m})\|)$.

	From such a weight-characteristic matrix we define the \emph{weight-quotient matrix} of $A$, with respect to
	$\mathcal{P}$, as $\tilde{B}^{*} \define (\tilde{S}^{*})^{\top}A\tilde{S}^{*}=(\tilde{b}^{*}_{ij})$.
	Notice that this matrix is symmetric and has entries
	\begin{align*}
		\tilde{b}^{*}_{ij}=\sum_{u,v \in
			V}\tilde{s}_{ui}^{*}a_{uv}\tilde{s}_{vj}^{*}=\sum_{u \in
			V_{i},v \in V_{j}}a_{uv}\nu_{u}\nu_{v}=\sum_{uv \in
			E(V_{i},V_{j})}\nu_{u}\nu_{v}=\tilde{b}^{*}_{ji},
	\end{align*}
	where $E(V_{i},V_{j})$ stands for the set of edges with ends in $V_{i}$ and $V_{j}$ (when $V_{i}=V_{j}$ each edge counts
	twice).

	In this article we will use the \emph{normalized weight-characteristic matrix} of $\mathcal{P}$, which is the $n\times m$ matrix $\bar{S}^{*}=(\bar{s}^{*}_{uj})$ with entries obtained by normalizing the columns of $\tilde{S}^{*}$, that is, $\bar{S}^{*}=\tilde{S}^{*}D^{-1}$.
Thus,
	\[
        \bar{s}^{*}_{uj}
        =
        \begin{cases}
          \frac{\nu_{u}}{\|\rho (V_{j})\|}, & \text{if } u\in V_{j},\\
          0, & \text{otherwise},
        \end{cases}
        \]	
	and it holds that $(\bar{S}^{*})^{\top}\bar{S}^{*}=I$.
    We define the \emph{normalized weight-quotient matrix} of $A$ with respect to $\mathcal{P}$, $\bar{B}^{*}=(\bar{b}^{*}_{ij})_{i,j \in [m]}$, as
    \begin{equation*}
      \bar{B}^{*}=(\bar{S}^{*})^{\top}A\bar{S}^{*}=D^{-1}(\tilde{S}^{*})^{\top}A\tilde{S}^{*}D^{-1}=D^{-1}\tilde{B}^{*}D^{-1},
    \end{equation*}	
    and hence $\bar{b}^{*}_{ij}=\frac{\tilde{b}^{*}_{ij}}{\|\rho (V_{i})\|\|\rho (V_{j})\|}.$
	
    \begin{table}
      \caption{Some particular cases of trivial partitions. Note that for weight-equitable partitions the coarsest partition is always trivial, while the finest partition is trivial for regular graphs.}\label{table:relations}
      \centering
      \begin{tabular*}{0.8\textwidth}{@{}l@{\;\;\extracolsep{\fill}}lr@{}}\toprule
        & \multicolumn{2}{c}{graph class admitting \dots partition with~$m$ cells}\\
        \cmidrule{2-3}
        number of cells~$m$ & equitable & weight-equitable\\
        \midrule
        1 & $\Longleftrightarrow$ regular & all\\
        2 & biregular & bipartite \\
        $n$ & all & $\Longleftrightarrow$ regular\\
        \bottomrule
      \end{tabular*}
    \end{table}
	
	In Table \ref{table:relations}, some trivial cases of (weight-) equitable partitions are summarized.
Note that for $m=2$, it does not hold that a partition into two sets is weight-equitable if and only if it is a bipartition of the graph;
there may be many other weight-equitable partitions and a graph which admits one may not be bipartite.
For example, the path graph $P_4$ on~4 vertices has two weight-equitable partitions:
its bipartition and the partition which groups the two endpoints and internal vertices.
However, a bipartition is always weight-equitable.

	The following characterization of weight-equitable partitions by the first author \cite{A2019} will be used to prove our main results.
	
	\begin{lemma}[\cite{A2019}]\label{lemma:aidafirstcharacterization}
		Let $A$ be the adjacency matrix of a connected graph $G$, and let $P$ be a weight-equitable partition of the vertex set of $G$ with normalized weight-characteristic matrix $\bar{S}^{*}$.
Then, $P$ is weight-equitable if and only if $A$ and $\bar{S}^{*}(\bar{S}^{*})^\top$ commute.\label{lem:commuteiffwr}
	\end{lemma}
	
    In \cite{A2019}, it is shown that weight-regular partitions can be used to improve the well-known Hoffman ratio bound on the chromatic number of a graph.
A graph coloring which satisfies this bound with equality is referred to as a \emph{Hoffman coloring}.
We take the opportunity to correct the statement of Proposition 5.3 (ii) in \cite{A2019}, which should say that if a graph $G$ has chromatic number $\chi(G)$ and a Hoffman coloring, then it holds that the multiplicity of the smallest eigenvalue $\lambda_n$ is at least $\chi(G)-1$ (and not only equal), and equality implies a unique Hoffman coloring.
	
	\section{Spectral Properties of Weight-Equitable Partitions}\label{sec:properties}

        As mentioned previously, the aim of weight-equitable partitions is to condense the adjacency matrix of an undirected graph to make assessing its spectrum easier.
        This section is devoted to, on the one hand, derive properties of the condensed adjacency matrix that are independent from the weight-equitable partition.
        That is, these results provide conditions that are necessary for a partition to be weight-equitable.
        On the other hand, we give new insights into the relation of weight-equitable and equitable partitions by
        providing a necessary criterion such that weight-equitable partitions are also equitable.
	
	\begin{theorem}\label{theo:Godsilbookcoro2.3extendedWR}
          Let~$G$ be a connected graph with adjacency matrix~$A$.
          Let $\lambda_{1}\geq \lambda_{2}\geq \cdots \geq \lambda_{n}$ be the eigenvalues of~$A$ and $\mathcal{P}$ be a weight-equitable partition for $A$.
          Let $\bar{B}^{*}$ be the normalized weight-quotient matrix of $A$ with respect to
          $\mathcal{P}$, with eigenvalues $\mu_{1}\geq \mu_{2}\geq \cdots
          \geq \mu_{m}$, $m< n$. Then, $\lambda_{1}=\mu_{1}$.
	\end{theorem}
	
	\begin{proof}
          Since $G$ is connected, $A$ is an irreducible matrix, which means that $B$ is also irreducible.
          Let $y=(y_{1},\ldots,y_{m})$ be a positive  eigenvector to $\mu_{1}$ and let $\bar{S}^{*}$ denote the normalized weight-characteristic matrix of $\Pa$.
          Note that Lemma 2.2 (a) from \cite{Cd40} can easily be extended to weight-equitable partitions if we replace $B$ and $P$ by $\bar{B}^{*}$ and $\bar{S}^{*}$.
          Then the vector $x=\bar{S}^{*}y$ is a positive vector such that $Ax=\mu_{1}x$, implying that $\mu_{1}$ is an eigenvalue of $A$ with eigenvector $x$.
          Perron-Frobenius Theorem implies that $\mu_{1}=\lambda_{1}$, completing the proof.
        \end{proof}

	As consequence of Theorem \ref{theo:Godsilbookcoro2.3extendedWR}, we obtain the analogous result for equitable partitions of graphs \cite[Corollary 2.3]{Cd40}, which states that the quotient matrix of an equitable partition has the same spectral radius as the adjacency matrix. We should observe that the above result also holds for non-negative symmetric matrices, see \cite[Theorem 2.1]{mscthesisaida}.

        Because of Theorem~\ref{theo:Godsilbookcoro2.3extendedWR} we know that the largest eigenvalue of~$A$ and~$\bar{B}^*$ are the same for any weight-equitable partition of~$V$.
        Based on this result, we can devise a necessary criterion for a partition to be weight-equitable that can be tested by evaluating a single matrix-vector multiplication.

	\begin{lemma}\label{lemma:masterthesis11}
          Let~$G$ be a connected graph and~$\{V_1,\dots,V_m\}$ be a weight-equitable partition of its vertex set.
          Let~$\lambda_1$ be the largest eigenvalue of~$G$.
          Then, the corresponding normalized weight-quotient matrix~$\bar{B}^*$ has
          eigenvector $x=\left(\|\rho(V_1)\|,\dots,\|\rho(V_m)\| \right)^{\top}$ with eigenvalue $\lambda_1$.
	\end{lemma}
	\begin{proof}
          For $i \in [n]$, the entries of $\bar{B}^*x$ equal
          \begin{align*}
            \left(\bar{B}^* x\right)_i = \sum_{j=1}^m \bar{b}_{ij}^*\|\rho(V_j)\|
            &= \sum_{j=1}^m \frac{\sum_{u\in V_i} \nu_u^2 b_{ij}^*(u)}{\| \rho(V_i)\|\|\rho(V_j)\|}\| \rho(V_j)\|\\
            &= \sum_{u\in V_i} \frac{\nu_u^2 }{\|\rho(V_i)\|}\sum_{j=1}^m b_{ij}^*(u)
            = \lambda_1 \|\rho(V_i)\|
            = \lambda_1 x_i.
          \end{align*}
	\end{proof}
	
	One of our main goals is to investigate means to find coarse weight-equitable partitions.
        Since every equitable partition is also weight-equitable (which follows trivially from~\cite[Lemma 2.2]{mscthesisaida}), we can use for example Bastert's algorithm~\cite{B1999} to compute a lower bound on the coarseness of a weight-equitable partition.
	One might thus wonder whether the converse can also be true, i.e., whether the coarsest non-trivial weight-equitable partition is also equitable.
        In general, this is not the case, but we are able to provide a necessary criterion.
	
	\begin{proposition}\label{lemma:weightequitableentriedPerron}
		Let $G$ be a connected graph with adjacency matrix $A$ and positive eigenvector~$\nu$,
		and consider a weight-equitable partition of the vertex set $\mathcal{P}=\{V_{1},\ldots,V_{m}\}$ with normalized weight-characteristic matrix $\bar{S}^\ast$.
                Then, $\nu=(\nu_1 \j^\top, \dots, \nu_m \j^\top)^\top$, with $\1$'s being all-one
                vectors of length $|V_i|$ for $i \in [m]$, if and only if $\mathcal{P}$ is equitable.
	\end{proposition}
	
\begin{proof}
		
Let $V_i,\ V_j\in \Pa$ and let $u,v\in V_i$ be arbitrary. Since $\Pa$ is weight-equitable, it must hold that
		\[b^{*}_{ij}(u) = \frac{1}{\nu_{u}}\sum_{w \in G(u)\cap
			V_{j}}\nu_{w} = \frac{1}{\nu_{v}}\sum_{w \in G(v)\cap
			V_{j}}\nu_{w} = b^{*}_{ij}(v).\]
If $\nu$ is constant over every cell of $\Pa$, this implies that $|G(u)\cap V_{j}| = |G(v)\cap V_{j}|$, hence $\Pa$ is equitable.
Conversely, let $\Pa$ be equitable with quotient matrix $B$ and characteristic matrix $S$.
It follows from \cite[Lemma 9.3.1]{GR2001} that $SB=AS$.
Then every eigenvector~$v$ of $B$ gives an eigenvector $S v =( v_1\j^\top, \dots , v_m \j^\top )^\top$ of $A$, since $B v = \lambda v$ implies that
		\[ A ( S v) = S B v = \lambda (S v).\]
In particular, the Perron eigenvector of $B$ gives the Perron eigenvector $\nu$ for $A$.
	\end{proof}
	
Note that, as a consequence of Proposition \ref{lemma:weightequitableentriedPerron}, all weight-equitable partitions of a regular graph are also equitable. For weight-equitable partitions which are not equitable, $\nu$ is not of the form requested in Proposition~\ref{lemma:weightequitableentriedPerron}, as illustrated in Example~\ref{ex:wrnotregular}.
	
	\begin{example} \label{ex:wrnotregular}
	The bipartite graph $G=(V_1\cup V_2, E)$ shown in Figure \ref{fig:nonregularWRpartition} has Perron eigenvector $\nu \approx (2.732, 1, 1, 1.414, 1.932, 1.932)$, which is not constant for either cell $V_i$.
However, it is easily checked that $\{V_1,V_2\}$ is a weight-equitable partition of $G$ (see also Table \ref{table:relations}).
		
		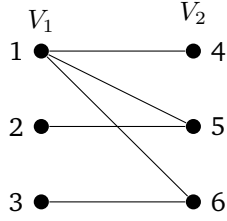
\begin{figure}[t]
			\centering
			\begin{tikzpicture}
				\node[draw=none,fill=none] (V1) at (0,0.4) {$V_1$};
				\node[draw=none,fill=none] (V2) at (2,0.5) {$V_2$};
				\node[shape=circle,fill=black,inner sep=2pt,label=left:1] (A) at (0,0) {};
				\node[shape=circle,fill=black,inner sep=2pt,label=right:4] (B) at (2,0) {};
				\node[shape=circle,fill=black,inner sep=2pt,label=left:2] (C) at (0,-1) {};
				\node[shape=circle,fill=black,inner sep=2pt,label=right:5] (D) at (2,-1) {};
				\node[shape=circle,fill=black,inner sep=2pt,label=left:3] (E) at (0,-2) {};
				\node[shape=circle,fill=black,inner sep=2pt,label=right:6] (F) at (2,-2) {};
				
				\draw[] (B) -- (A) -- (D) -- (C);
				\draw[] (A) -- (F) -- (E);
			\end{tikzpicture}
			\caption{A weight-equitable bipartition such that Perron eigenvector $\nu$ is not constant over each cell.}\label{fig:nonregularWRpartition}
		\end{figure}

	\end{example}
	
	\begin{remark}
		A vertex partition of a graph $G$ is called an \emph{orbit partition} if its classes correspond to the orbits of (a subgroup of) $\text{Aut}(G)$.
A graph is said to be \emph{compact} if every doubly stochastic matrix which commutes with its adjacency matrix $A$ is a convex combination of permutation matrices that commute with $A$.
Godsil \cite{G1997} showed that for compact graphs, all equitable partitions are orbit partitions.
This result does not extend to weight-equitable partitions, as the trivial partition $\{V\}$ is always weight-equitable, see Table \ref{table:relations}.
Therefore, $G$ must be vertex-transitive, and hence regular, if all weight-equitable partitions are also orbit partitions.
In that case, $\nu$ is constant, so it follows from Proposition \ref{lemma:weightequitableentriedPerron} that each weight-equitable partition is actually equitable.
	\end{remark}

	\section{Characterizations of Weight-Equitable Partitions}\label{sec:characterizations}
	
        In contrast to equitable partitions, no algorithmic procedure to find weight-equitable partitions has been discussed in the literature.
        To make progress in this direction, we derive novel characterizations of weight-equitable partitions and methods to generate weight-equitable partitions from known ones.
        We will investigate one of these methods from a practical point of view in Section~\ref{sec:algo}.

        Let~$A \in \R^{m \times n}$.
        A pair of doubly stochastic matrices~$(X, Y) \in \R^{m \times m} \times \R^{n \times n}$ is called a \emph{fractional automorphism} of~$A$ if $XA = AY$.
        A fractional automorphism of type~$(X, X)$ is called a \emph{fractional isomorphism}.
        Note that these definitions generalize the concept of graph automorphisms and isomorphisms from permutation matrices to doubly stochastic matrices.

        In our first characterization, we link weight-equitable partitions of~$G$ to fractional isomorphisms of its adjacency matrix.
        Given a partition $\mathcal{P}$ of $[n]$, we define $X_{\mathcal{P}}\in \R^{n\times n}$ to be the matrix with entries $x_{vv'} \define \frac{\nu_v\nu_{v'}}{\|\rho (P)\|^2}$ if $v, v' \in P$ for some $P \in \mathcal{P}$ and $x_{vv'} \define 0$ otherwise.
	
        \begin{proposition}\label{thn:firstcharacterization}\label{th:weak2}
        If $\mathcal{P}$ is a weight-equitable partition of $V$, then $X_{\mathcal{P}}A = AX_{\mathcal{P}}$.
	   \end{proposition}
	
	\begin{proof}
		This follows from Lemma \ref{lemma:aidafirstcharacterization}, using the fact that $X_{\mathcal{P}} = \bar{S}^{*}(\bar{S}^{*})^\top$.		
	\end{proof}
		
Proposition \ref{th:weak2} shows commutativity of a matrix derived from a weight-equitable partition.
For equitable partitions, Godsil \cite[Theorem 1.5]{G1997} considered the converse. To this end, for a double stochastic matrix~$X \in \R^{n \times
          n}$,
        define the directed graph~$G_X$ on~$n$ vertices with adjacency matrix
        $A = (A_{ij})_{i,j \in [n]}$, where
        \[
          A_{ij}=
          \begin{cases}
            1, & \text{if } X_{ij} \neq 0,\\
            0, & \text{otherwise.}
          \end{cases}
        \]
        The \emph{strongly connected components} of $X$ are defined as the strongly connected components of $G_X$.
        Let $\mathcal{P}_X$ denote the partition of $[n]$ into the strongly connected components of $X$.
       Since every equitable partition is also weight-equitable, Godsil's result also trivially applies to weight-equitable partitions:

	\begin{theorem}\label{theo:godsilthm1.5}
		If $X$ is a doubly stochastic matrix which commutes with $A(G)$, then the partition $\mathcal{P}_X$ is weight-equitable.\label{th:godsil1.5}
	\end{theorem}

	For equitable partitions, Theorem \ref{theo:godsilthm1.5} implies that $I$ is the only doubly stochastic matrix which commutes with $A(G)$ if and only if $G$ has no nontrivial equitable partitions \cite[Corollary 1.6]{G1997}.
In the case of weight-equitability, we cannot generalize Theorem \ref{theo:godsilthm1.5} without extra assumptions, as the trivial partition is not necessarily weight-equitable, see Table~\ref{table:relations}.
As a partial extension of Theorem~\ref{theo:godsilthm1.5}, Proposition~\ref{thn:firstcharacterization} implies the following result.
	
	\begin{corollary}\label{cor:doublypermute}
		If $I$ is the only doubly stochastic matrix which commutes with $A(G)$, then the trivial partition of $n$ cells is the only potential (since the graph may not be regular) weight-equitable partition of $G$.
	\end{corollary}
	Equivalently, if $I$ is the only doubly stochastic matrix which commutes with $A(G)$, then $G$ has a weight-equitable partition if and only if it is a regular graph.

        In order to state our next result we need some preliminary definitions. Let $A\in \R^{V^0\times W^0}$ and~$B\in\R^{V^1\times W^1}$ be two real-valued matrices. Two matrices $A$ and $B$ are said to be \emph{fractionally isomorphic} if the following three properties hold:
        \begin{itemize}
        \item there are doubly stochastic matrices~$X$, $Y$ such that $X(A\oplus B) = (A \oplus B)Y$;

        \item for every~$i \in \{0,1\}$ and all~$v \in V^i$, there exists~$v' \in V^{1-i}$ with~$x_{vv'} \neq 0$;
        \item for every~$i \in \{0,1\}$ and all~$w \in W^i$, there exists~$w' \in V^{1-i}$ with~$x_{ww'} \neq 0$.
        \end{itemize}
        A \emph{joint partition} of graphs $G$ and $H$ is a partition of $V(G)\cup V(H)$.
        A joint partition is \emph{balanced} if every part has nonempty intersection with both $V(G)$ and $V(H)$.
	
	The following result extends \cite[Theorem 5.2]{GKMS2014} for weight-equitable partitions.
	
\begin{theorem}\label{thm:join} For all graphs $G$ and $H$ with adjacency matrices $A$ and $B$, respectively, the following two statements are equivalent:
		\begin{enumerate}[(i)]
			\item\label{it:join1} $G$ and $H$ have a balanced weight-equitable joint partition;
			\item\label{it:join2} The coarsest weight-equitable joint partition of $G$ and $H$ is balanced.
                \end{enumerate}
                Moreover, if $\rho\colon u \rightarrow \nu_u$ is constant over each cell of the partition,  then (\ref{it:join1}) and (\ref{it:join2}) imply that $A$ and~$B$ are fractionally isomorphic.
	\end{theorem}

	\begin{proof}
		Statement~\eqref{it:join1} follows immediately from~\eqref{it:join2}.
Conversely, if there exists a balanced weight-equitable joint partition $\mathcal{P}$, it is a refinement of the coarsest one, hence any part $P$ of the coarsest weight-equitable joint partition is the union of some $P_1, \dots, P_s \in \mathcal{P}$.
Since each part $P_i$ has nonempty intersection with $V(G)$ and $V(H)$, so does $P$.
		
		To show that the above imply that $A$ and $B$ are fractionally isomorphic, let $\mathcal{P}=\{P_1, \dots, P_s\}$ be a balanced weight-equitable joint partition of $G$ and $H$.
By Lemma \ref{lem:commuteiffwr}, the matrix $\bar{S}^{*}(\bar{S}^{*})^\top$ satisfies $\bar{S}^{*}(\bar{S}^{*})^\top(A\oplus B) = (A\oplus B)\bar{S}^{*}(\bar{S}^{*})^\top$.
Set $X=Y=\bar{S}^{*}(\bar{S}^{*})^\top$ and let $P_i\in \mathcal{P}$, $v\in P_i$. Since we are assuming  that  $\rho\colon u \rightarrow \nu_u$ is constant over each cell of the partition, it holds that $X\1=\1X=\1$.
Assume without loss of generality that $v\in V(G)$.
Since $\mathcal{P}$ is balanced, there exists a vertex $v'\in P_i$ such that $v'\in V(H)$.
This means that $(\bar{S}^{*}(\bar{S}^{*})^\top)_{vv'}\neq 0$.
Therefore $(X,Y)$ is a fractional isomorphism from $A$ to $B$.
	\end{proof}
	
	Observe that in order to link weight-equitable partitions to fractional isomorphism as we do in the last part of Theorem \ref{thm:join}, one cannot avoid the assumption that $\rho\colon u \rightarrow \nu_u$ is constant over each cell of the partition, and actually this case is just equivalent to the equitable partition characterization that appeared in \cite[Theorem 5.2]{GKMS2014}. Note that this is closely related with the result of Proposition \ref{lemma:weightequitableentriedPerron}. This is due to the fact that if a weight-equitable partition is equitable, then we know that $XA=AX, X\1=\1X=\1, X\geq 0$, see \cite[Corollary 1.2]{G1997}. If a weight-equitable partition is not equitable, then one can only guarantee $AX=XA, X\1=\1X, X\geq 0$ \cite{A2019}. To the best of our knowledge, it remains an open problem to investigate the convex polytope that consists of all matrices $X$ such that $XA=AX,X\1=\1X,X\geq 0$.

Let $\Pa$ be a joint partition of graphs $G$ and $H$. The \emph{restriction} of $\Pa$ to $G$ is defined as~$\Pa_G\define \{P\cap V(G) \mid P\in \Pa\}$. We denote the intersection of a cell~$P\in \Pa$ with~$V(G)$ and~$V(H)$ by~$P_G$ and~$P_H$ respectively.

 \begin{proposition}\label{extension5.4Groheetal}
    Let $G$ and $H$ be graphs with coarsest weight-equitable joint partition $\Pa$. Then the restrictions of $\Pa$ to $G$ and $H$ are the coarsest weight-equitable partitions of $G$ and $H$.\label{prop:restrictionCoarsest}
\end{proposition}
\begin{proof}
  Without loss of generality, consider the restriction $\Pa_G$.
Let $P,Q\in \Pa$ and fix an arbitrary vertex~$u\in P_G$, then
\[b^*_{P,Q}(u) = \sum_{\substack{v\sim u \\ v\in Q}}\frac{\nu_v}{\nu_u} = \sum_{\substack{v\sim u \\ v\in Q_G}}\frac{\nu_v}{\nu_u} = b^*_{P_G, Q_G}(u).\]
Since $\Pa$ is a weight-equitable partition, $b^*_{P,Q}(u)~$ does not depend on $u$, hence $\Pa_G$ is weight-equitable.

For the sake of contradiction, suppose that~$G$ has a coarser weight-equitable partition~$\Qa$.
Let~$\Qa'$ be the partition of~$V(G)\cup V(H)$ which is given by
\[\Qa'=\left\{\bigcup_{\substack{P\in \Pa\\ \ P_G\subseteq Q}}P \mid Q\in \Qa\right\}.\]
and consider two arbitrary cells $P,Q\in \Qa'$.
For a vertex~$u\in P_G$, the weight-intersection number is given by~$b^*_{P,Q}(u) = b^*_{P_G,Q_G}(u)$,
which is independent of~$u$, since~$\Qa$ is weight-equitable.
If $u\in P_H$, let $S\subseteq P$ denote the cell of $\Pa$ containing $u$.
Then
\[b^*_{P,Q}(u) = b^*_{P_H,Q_H}(u) = \sum_{\substack{T\subseteq Q\\T\in \Pa}} b^*_{S_H,T_H}(u).\]
Note that for any~$x\in S_G$, we have~$b^*_{S_H,T_H}(u) = b^*_{S_G,T_G}(x)$, because~$\Pa$ is weight-equitable.
It follows that \[b^*_{P,Q}(u) =  \sum_{\substack{T\subseteq Q\\T\in \Pa}} b^*_{S_G,T_G}(x) = b^*_{S_G,Q_G}(x).\]
By weight-equitability of $\Qa$, this number again equals $b^*_{P_G,Q_G}(x)$ and is independent of our choice of $x$ and $u$.
Therefore, $\Qa'$ is a coarser weight-equitable partition, a contradiction.
\end{proof}

\begin{proposition}\label{extension5.6Groheetal}
Let $G$ and $H$ be graphs with adjacency matrices $A$ and $B$ and balanced weight-equitable
joint partition $\Pa$.
If $G$ is connected, then for all $P\in \Pa$,
\[\frac{\|\rho (P_G)\|^2}{\|\rho (P_H)\|^2} = \frac{\|\rho (V(G))\|^2}{\|\rho (V(H))\|^2}.\]
\end{proposition}
\begin{proof}
Fix $P\in \Pa$ arbitrarily. Since $G$ is connected, there exists a cell $Q\in \Pa$ such that $b^*_{P,Q}\neq 0$.
We have shown in Proposition \ref{prop:restrictionCoarsest} that the restriction $\Pa_G$ is again a weight-equitable partition with the same weight-intersection numbers.
As a result,
\[b^*_{P,Q}\|\rho (P_G)\|^2 = b^*_{P_G,Q_G} \|\rho (P_G)\|^2 = \tilde{b}^*_{P,Q} = \tilde{b}^*_{Q,P} = b^*_{Q_G,P_G} \|\rho (Q_G)\|^2 = b^*_{Q,P}\|\rho (Q_G)\|^2.\]
Similarly, we can derive $b^*_{P,Q}\|\rho (P_H)\|^2 = b^*_{Q,P}\|\rho (Q_H)\|^2$.
Combining both statements gives
\[\frac{b^*_{Q,P}}{b^*_{P,Q}} = \frac{\|\rho (P_G)\|^2}{\|\rho (Q_G)\|^2}= \frac{\|\rho (P_H)\|^2}{\|\rho (Q_H)\|^2} ,\]
which rewrites to
\begin{equation}\frac{\|\rho (P_G)\|^2}{\|\rho (P_H)\|^2}= \frac{\|\rho (Q_G)\|^2}{\|\rho (Q_H)\|^2} .\label{eq:relationbij}\end{equation}

Consider the graph $G'$ on vertex set $\Pa$ with edges $\{PQ \mid P, Q\in \Pa,\ b^*_{P,Q}\neq 0\}$.
This is a connected graph, since $G$ is connected.
Note that the relation in Equation \eqref{eq:relationbij} is transitive: if $R\in \Pa$ such that $b^*_{Q,R}\neq 0$, then we have
\[\frac{\|\rho (P_G)\|^2}{\|\rho (P_H)\|^2}= \frac{\|\rho (Q_G)\|^2}{\|\rho (Q_H)\|^2} = \frac{\|\rho (R_G)\|^2}{\|\rho (R_H)\|^2}.\]
Since $G'$ is connected, this means that Equation \eqref{eq:relationbij} holds for any $P,Q\in \Pa$.
It follows that
\[\sum_{Q\in \Pa} \|\rho (Q_G)\|^2 \|\rho (P_H)\|^2= \sum_{Q\in \Pa} \|\rho (Q_H)\|^2 \|\rho (P_G)\|^2,\]
hence
\[\frac{\|\rho (P_G)\|^2}{\|\rho (P_H)\|^2} = \frac{\|\rho (V(G))\|^2}{\|\rho (V(H))\|^2}.\]

\end{proof}

If one restrict to equitable partitions, then
 propositions \ref{extension5.4Groheetal} and \ref{extension5.6Groheetal} give \cite[Lemma 5.4]{GKMS2014} and \cite[Lemma 5.6]{GKMS2014}, respectively.

	For stating the next characterization, we first need to introduce a linear operator. Let $G$ be a connected graph with Perron eigenvector $\nu$. Let $\mathds{V}$ be the vector space of all real functions on~$V$, and, for any partition~$\Pa$ of~$V$, let~$F(V,\Pa)$ be the subspace  of~$\mathds{V}$ that consists of all functions on~$V$ that are constant on the cells of~$\Pa$. Consider the linear operator~$\mathcal{B}\colon \mathds{V} \to \mathds{V}$ defined by

	\[(\mathcal{B}f)(u) \define \sum_{v\sim u} \frac{\nu_v}{\nu_u}f(u).\]
	
	\begin{lemma} Let $G$ be a connected graph and let $\Pa=\{V_1,\dots,V_m\}$ be a partition of $V$.
          Then, $\Pa$ is weight-equitable if and only if $F(V,\Pa)$ is $\mathcal{B}$-invariant.
		\label{lemma:changodsil1}
	\end{lemma}
	\begin{proof}
The space $F(V,\Pa)$ is spanned by the characteristic functions $x_1,\dots, x_m$ which are one on cell $1,\dots,m$ respectively and zero elsewhere.
It therefore suffices to show $\mathcal{B}$-invariance of this basis.
Let $i,j\in[m]$ and $u\in V_i$.
Then
		\[(\mathcal{B}x_j)(u) = \sum_{v\sim u}\frac{\nu_v}{\nu_u} x_j(v) = \frac{1}{\nu_u}\sum_{v\in V_j\cap G(u)}\nu_v = b_{ij}^*(u).\]
This means that $\mathcal{B}x_j$ is constant on $V_i$ if and only if $b_{ij}^*(u)$ is independent of $u$.
Hence $\mathcal{B}x_j\in F(V,\Pa)$ if and only if $\Pa$ is weight-equitable.
	\end{proof}
	
	The set of partitions of $V$ can be viewed as a lattice where $\Qa\le \Pa$ if $\Qa$ is a refinement of $\Pa$.
 The \emph{meet} $\Pa \land \Qa$ of two partitions $\Pa$, $\Qa$ is the partition whose cells are the nonempty pairwise intersections of cells of $\Pa$ and $\Qa$.
  The \emph{join} $\Pa\vee\Qa$  is given by the connected components of the graph with edge set $\{\{u,v\} : u, v \text{ share a cell in } \Pa \text{ or } \Qa\}$.

	\begin{corollary} Let $\Pa$ and $\Qa$ be weight-equitable partitions of a connected graph $G$. Then the join $\Pa \vee \Qa$ is a weight-equitable partition.
		\label{lemma:changodsil2}
	\end{corollary}
	\begin{proof}
		Since $\Pa$ and $\Qa$ are both weight-equitable, Lemma \ref{lemma:changodsil1} implies that $F(V,\Pa)$ and $F(V,\Qa)$ are both $\mathcal{B}$-invariant.
It then follows from \cite[Lemma 2.3]{CG1997} that $F(V,\Pa\vee\Qa)$ is also $\mathcal{B}$-invariant, so $\Pa\vee\Qa$ is weight-equitable.
	\end{proof}
		
	As a result of Lemma \ref{lemma:changodsil1} and Corollary \ref{lemma:changodsil2}, we obtain the equitable partition results from \cite[Section 5]{CG1997}.
 Also, observe that as a consequence of Corollary \ref{lemma:changodsil2}, any partition $\Pa$ has a unique maximal weight-equitable refinement, which is given by the join of all the weight-equitable partitions that refine $\Pa$.
 Contrary to the join, the meet of two weight-equitable partitions need not be weight-equitable, as illustrated in Example \ref{ex:meetNotWR}.
	
	\begin{figure}[t]
		\centering
		\begin{tikzpicture}
			\node[draw=none] (A1) at (0,0.5) {$1$};
			\node[draw=none] (B1) at (1,0.5) {$2$};
			\node[draw=none] (C1) at (2,0.5) {$3$};
			\node[draw=none] (D1) at (0,-1.5) {$4$};
			\node[draw=none] (E1) at (1,-1.5) {$5$};
			\node[draw=none] (F1) at (2,-1.5) {$6$};
			\node[shape=circle,fill=black,inner sep=2pt] (A) at (0,0) {};
			\node[shape=circle,fill=black,inner sep=2pt] (B) at (1,0) {};
			\node[shape=circle,fill=black,inner sep=2pt] (C) at (2,0) {};
			\node[shape=circle,fill=black,inner sep=2pt] (D) at (0,-1) {};
			\node[shape=circle,fill=black,inner sep=2pt] (E) at (1,-1) {};
			\node[shape=circle,fill=black,inner sep=2pt] (F) at (2,-1) {};
			
			\draw[] (A) -- (B) -- (C) -- (F) -- (E) -- (D) -- (A);
			\draw[] (D) -- (B) -- (E) -- (C);
			\draw[] (A) -- (E);
			\draw[] (B) -- (F);
		\end{tikzpicture}
		\caption{A graph with two weight-equitable partitions, $\{\{1,3,5\},\{2,4,6\}\}$ and $\{\{1,5,6\},\{2,3,4\}\}$, whose meet is not weight-equitable.}\label{fig:meetNotWR}
	\end{figure}
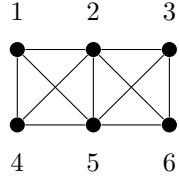
	
	\begin{example} Consider the graph $G$ given in Figure \ref{fig:meetNotWR} with Perron eigenvector $(1,\sqrt{2},1,1,\sqrt{2},1)$.
The partitions $\Pa = \{\{1,3,5\},\{2,4,6\}\}$ and $\Qa = \{\{1,5,6\},\{2,3,4\}\}$ of $G$ are weight-equitable, but not equitable.
However, the meet $\Pa\land \Qa = \{\{1,5\},\{3\},\{6\},\{2,4\}\}$ is not a weight-equitable partition.
		\label{ex:meetNotWR}
	\end{example}

		In Section \ref{sec:experiments} we will show an application of the Corollary \ref{lemma:changodsil2} for computing coarse weight-equitable partitions of cographs.
	
	\section{Computational Aspects of Weight-Equitable Partitions}\label{sec:algo}

        As mentioned in the introduction, there are many cases in which one
        is interested in finding coarse weight-equitable partitions. While it is known that the coarsest equitable partition of a graph can be found in polynomial time, see for example Corneil et al.~\cite{CC} and Bastert~\cite{B1999}, nothing seems to be known about computing weight-equitable partitions. For the latter, Bastert's approach~\cite{B1999} becomes trivial, since the initialization step of his algorithm includes all vertices in the same cell, a case which is always weight-equitable (see Table \ref{table:relations}). Thus, Bastert's algorithm cannot be generalized to weight-equitable partitions immediately.

	The aim of this section is to investigate the potential of the join
        operator derived in the previous section in generating coarse
        partitions from finer ones.
        Moreover, a natural choice for a fine partition is to study
        partitions all of whose cells have the same size~$c$.
        We call such a partition \emph{$c$-homogeneous}, where the
        parameter~$c$ controls the coarseness of the partition.

        To study the potential of the join operator systematically,
        however, we need to be able to efficiently generate fine
        weight-equitable partitions.
        Since finding equitable 2-homogeneous partitions is NP-hard
        (Section~\ref{sec:complexity}), we focus in our investigation on
        the join operator on cographs.
        We will see that for such graphs the concept of equitability and
        weight-equitability coincides for 2-homogeneous partitions and that we can find
        (weight-) equitable 2-homogeneous partitions very efficiently
        (Section~\ref{sec:cographs}).
        Based on these results, Section~\ref{sec:experiments} investigates
        the capability of the join operator to find coarse weight-equitable
        partitions.
	
	\subsection{Finding Fine Equitable Partitions is Hard}\label{sec:complexity}
	
	Let~$G$ be an undirected graph.
	An \emph{automorphism} of~$G$ is a bijection~$\gamma\colon V \to V$ that preserves
	adjacency, i.e., $\{\gamma(u),\gamma(v)\} \in E$ if and only if~$\{u,v\}
	\in E$.
        An automorphism $\gamma$ is \emph{fixed-point-free} if there is no
        vertex~$v$ such that $\gamma(v) = v$.
        The \emph{order} of $\gamma$ is
        the smallest positive integer $i$ such that $\gamma^i$ is the
        identity. An automorphism of order two is called an
        \emph{involution}.
	
	Lubiw~\cite{Lubiw1981} studied the complexity of several
        algorithmic problems related to graph automorphism.
        In particular, she has shown that deciding whether a given graph $G$ has a fixed-point-free automorphism of order two is NP-complete.
        With the following observation, we can link this to the complexity
        of finding certain equitable partitions.
	
	\begin{lemma}
		\label{lem:automorphismIFFpartitionRegular}
		Let~$G$ be an undirected graph.
		Then, $G$ has an automorphism being an involution without fixed points if
		and only if~$G$ admits an equitable partition
		with~$\frac{n}{2}$ cells each having size~2.
	\end{lemma}
	\begin{proof}
		Suppose~$G$ has an equitable partition~$\{V_1, \dots, V_m\}$ with~$m
		= \frac{n}{2}$ and~$\card{V_i} = 2$ for all~$i \in [m]$.
		For every~$i \in [m]$, we assume that~$V_i$ is given by~$\{i_1, i_2\}$.
		To prove the first part of the assertion, we show that~$\gamma = \prod_{i = 1}^m
		(i_1, i_2)$ is an automorphism of~$G$ that is an involution without
		fixed points.
		Since~$\{V_1, \dots, V_m\}$ is a partition of~$V$ and all cells have
		cardinality~2, $\gamma$ is an involution without fixed points.
		Hence, it remains to show that~$\gamma$ is an automorphism of~$G$.
		
		Let~$i, j \in [m]$ and~$r,s \in [2]$.
		Moreover, let~$r'$ and~$s'$ be the complementary indices of~$r$ and~$s$
		in~$[2]$, respectively.
		To show that~$\gamma$ is an automorphism, we need to show that~$e
		= \{i_{r}, j_{s}\} \in E$ if and only if~$\gamma(e) = \{i_{r'},
		j_{s'}\} \in E$.
		Observe that both~$i_r$ and~$i_{r'}$ have the same number of neighbors
		in~$U_j$, since $\{V_1, \dots, V_m\}$ is equitable.
		Consequently, if~$i_r$ is not adjacent with~$U_j$, also~$i_{r'}$ is not,
		which implies~$e, \gamma(e) \notin E$.
		Furthermore, if~$i_r$ is adjacent with~$j_{s}$ and~$j_{s'}$, so
		is~$i_{r'}$.
		Hence, $e, \gamma(e) \in E$.
		Finally, it remains to consider the case that~$i_r$ and~$i_{r'}$
		are adjacent with exactly one vertex in~$U_j$.
		Then, both cannot be adjacent with the same vertex in~$U_j$, because
		otherwise one vertex in~$U_j$ would have two neighbors in~$U_i$ while the
		other has no neighbor in~$U_i$, contradicting equitability.
		This again implies~$e \in E$ if and only if~$\gamma(e) \in E$, concluding
		the first part of the proof.
		
		For the reverse direction, let $\gamma = \prod_{i = 1}^m (i_1, i_2)$ be an involutionary
		automorphism of~$G$ without fixed points. Then $\{V_1, \dots, V_m\}$, where, for
		every~$i \in [m]$, $V_i = \{i_1, i_2\}$, is a partition of~$V$.
		As~$\gamma$ is an automorphism   of~$G$, each node in~$V_i$ has the same number of neighbors in~$V_j$ for all~$i,j \in [m]$. Hence $\{V_1, \dots, V_m\}$ is an equitable partition.
	\end{proof}
	
	Due to the aforementioned result by Lubiw~\cite{Lubiw1981}, we thus
	conclude the following result.
	
	\begin{corollary}
		\label{cor:partitionNPhard}
        Deciding whether a given graph $G$ admits an equitable partition with $\frac{n}{2}$ cells is NP-complete.
	\end{corollary}

	\subsection{Finding Weight-Equitable Partitions for Cographs}\label{sec:cographs}
	
	Corollary \ref{cor:partitionNPhard} shows that, unless $\text{P} =
        \text{NP}$, we can generally not decide in polynomial time whether a graph admits an equitable partition with cells of size two.
        However, for certain graph families we can exploit the graph
        structure to obtain an efficient algorithm to compute such partitions.
        We will show that for the class of cographs the existence of these
        specific equitable partitions can be decided in polynomial time.
        \begin{definition}
          An undirected simple graph is called \emph{cograph} if it does
          not contain an induced $P_4$.
        \end{definition}
    	Cographs have been studied extensively in the literature. From the spectral point of view, Jung~\cite{J} introduced an algorithm for locating eigenvalues of cographs in a given interval. Ghorbani~\cite{G2018} provided a new characterization of cographs, and further properties of the eigenvalues (of the adjacency matrix) of a cograph were explored, e.g., by Ghorbani~\cite{G2019}, Mohammadian and Trevisan~\cite{MT2016}, and Jacobs et al.~\cite{JTT2018}.

        Before we describe an algorithm to find 2-homogeneous (weight-)
        equitable partitions of a cograph~$G$, we first show that there is
        indeed no difference between 2-homogeneous weight-equitable
        partitions and equitable partitions.
	\begin{lemma}
          \label{lem:automorphismIFFpartitionWeightRegular}
          Let~$G$ be a connected cograph. Then, every
          weight-equitable partition with~$\frac{n}{2}$ cells each
          having size~2 of~$G$ is an equitable partition of~$G$.
	\end{lemma}
	\begin{proof}
          Suppose~$G$ has a weight-equitable partition~$\{V_1, \dots, V_m\}$ with~$m  = \frac{n}{2}$ and~$\card{V_i} = 2$ for every~$i \in [m]$. For every~$i \in [m]$, we assume that~$V_i$ is given by~$\{i_1, i_2\}$. Consider two arbitrary cells $V_i$ and $V_j$ and their induced subgraph. Note that if $i_1$ has no neighbors in $V_j$, then neither does $i_2$ and vice versa. As we show next, $i_1$ and $i_2$ must have the same number of neighbors in $V_j$.

          \begin{figure}[t]
            \centering
            \begin{tikzpicture}
              \node[shape=circle,fill=black,inner sep=2pt] (A1) at (0,0) {};
              \node[shape=circle,fill=black,inner sep=2pt] (B1) at (1,0) {};
              \node[shape=circle,fill=black,inner sep=2pt] (C1) at (1,1) {};
              \node[shape=circle,fill=black,inner sep=2pt] (D1) at (0,1) {};
              \node[fill=none,draw=none] (E1) at (0.5,-0.5) {(a)};
              \draw[] (A1)--(B1)--(C1)--(D1)--(A1);
              \draw[black,thick,dotted] (-0.2,0.5)--(1.2,0.5);

              \node[shape=circle,fill=black,inner sep=2pt] (A2) at (2,0) {};
              \node[shape=circle,fill=black,inner sep=2pt] (B2) at (3,0) {};
              \node[shape=circle,fill=black,inner sep=2pt] (C2) at (3,1) {};
              \node[shape=circle,fill=black,inner sep=2pt] (D2) at (2,1) {};
              \node[draw=none] (E2) at (2.5,-0.5) {(b)};
              \draw[] (A2)--(D2)--(B2)--(C2)--(A2);
              \draw[black,thick,dotted] (1.8,0.5)--(3.2,0.5);

              \node[shape=circle,fill=black,inner sep=2pt] (A3) at (4,0) {};
              \node[shape=circle,fill=black,inner sep=2pt] (B3) at (5,0) {};
              \node[shape=circle,fill=black,inner sep=2pt] (C3) at (5,1) {};
              \node[shape=circle,fill=black,inner sep=2pt] (D3) at (4,1) {};
              \node[draw=none] (E3) at (4.5,-0.5) {(c)};
              \draw[] (A3)--(B3)--(C3)--(D3)--(A3);
              \draw[] (A3)--(C3);
              \draw[] (B3)--(D3);
              \draw[black,thick,dotted] (3.8,0.5)--(5.2,0.5);

              \node[shape=circle,fill=black,inner sep=2pt] (A4) at (6,0) {};
              \node[shape=circle,fill=black,inner sep=2pt] (B4) at (7,0) {};
              \node[shape=circle,fill=black,inner sep=2pt] (C4) at (7,1) {};
              \node[shape=circle,fill=black,inner sep=2pt] (D4) at (6,1) {};
              \node[draw=none] (E4) at (6.5,-0.5) {(d)};
              \draw[] (B4)--(C4)--(D4)--(A4);
              \draw[] (A4)--(C4);
              \draw[] (B4)--(D4);
              \draw[black,thick,dotted] (5.8,0.5)--(7.2,0.5);

              \node[shape=circle,fill=black,inner sep=2pt] (A5) at (8,0) {};
              \node[shape=circle,fill=black,inner sep=2pt] (B5) at (9,0) {};
              \node[shape=circle,fill=black,inner sep=2pt] (C5) at (9,1) {};
              \node[shape=circle,fill=black,inner sep=2pt] (D5) at (8,1) {};
              \node[draw=none] (E5) at (8.5,-0.5) {(e)};
              \draw[] (A5)--(B5);
              \draw[] (C5)--(D5);
              \draw[black,thick,dotted] (7.8,0.5)--(9.2,0.5);

              \node[shape=circle,fill=black,inner sep=2pt] (A6) at (10,0) {};
              \node[shape=circle,fill=black,inner sep=2pt] (B6) at (11,0) {};
              \node[shape=circle,fill=black,inner sep=2pt] (C6) at (11,1) {};
              \node[shape=circle,fill=black,inner sep=2pt] (D6) at (10,1) {};
              \node[draw=none] (E6) at (10.5,-0.5) {(f)};
              \draw[black,thick,dotted] (9.8,0.5)--(11.2,0.5);

              \node[shape=circle,fill=black,inner sep=2pt] (A7) at (12,0) {};
              \node[shape=circle,fill=black,inner sep=2pt] (B7) at (13,0) {};
              \node[shape=circle,fill=black,inner sep=2pt] (C7) at (13,1) {};
              \node[shape=circle,fill=black,inner sep=2pt] (D7) at (12,1) {};
              \node[draw=none] (E7) at (12.5,-0.5) {(g)};
              \draw[] (A7)--(D7);
              \draw[] (B7)--(C7);
              \draw[black,thick,dotted] (11.8,0.5)--(13.2,0.5);

            \end{tikzpicture}
            \caption{All possible subgraphs induced by two cells of a
              weight-equitable partition of a cograph with cells of size
              two.
              The dotted line separates the two cells.
            }\label{fig:subgraph2cells}
          \end{figure}
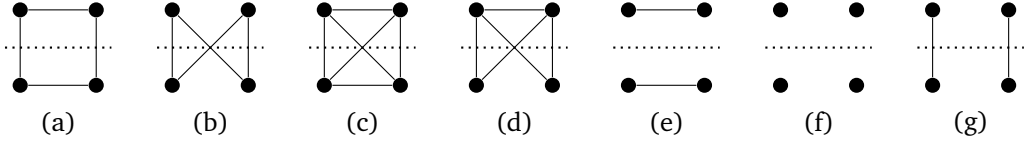

          Without loss of generality, assume that $i_1$ is adjacent to both $j_1$ and $j_2$ and $i_2$ only to $j_1$. If $i_1\sim i_2$, then weight-equitability ensures that $\nu_{i_1}=\nu_{i_2}$, as $b^{*}_{ii}(i_1) = b^{*}_{ii}(i_2)$. At the same time, it must hold that $b^*_{ij}(i_1) = b^*_{ij}(i_2)$. This implies that $\nu_{j_1} = \nu_{j_1}+\nu_{j_2}$, contradicting the fact that $\nu > 0$. If $i_1\not\sim i_2$ and $j_1\not\sim j_2$, then $V_i$ and $V_j$ induce a path of length three, which contradicts the fact that $G$ is cograph. The case $j_1\sim j_2$ is symmetric to $i_1\sim i_2$, $j_1\not\sim j_2$. This means that vertices that share a cell have an equal number of neighbors in every other cell.

          Figure \ref{fig:subgraph2cells} shows (up to symmetry) all possible induced subgraphs of two cells that satisfy the above requirement. In graph (a)-(e), weight-equitability directly implies that the Perron eigenvector is constant over both cells. We will show that this is also the case for (f) and (g). By Proposition \ref{lemma:weightequitableentriedPerron}, a weight-equitable partition with constant Perron entries over each part is equitable, so the result follows.

          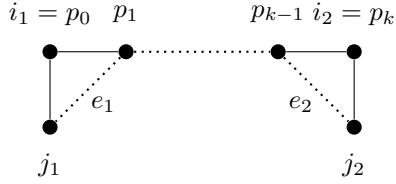
\begin{figure}[t]
            \centering
            \begin{tikzpicture}
              \node[shape=circle,fill=black,inner sep=2pt] (A) at (0,0) {};
              \node[shape=circle,fill=black,inner sep=2pt] (B) at (1,0) {};
              \node[shape=circle,fill=black,inner sep=2pt] (C) at (3,0) {};
              \node[shape=circle,fill=black,inner sep=2pt] (D) at (4,0) {};
              \node[shape=circle,fill=black,inner sep=2pt] (E) at (0,-1) {};
              \node[shape=circle,fill=black,inner sep=2pt] (F) at (4,-1) {};
              \node[fill=none,draw=none] (i1) at (0,0.5) {$i_1=p_0$};
              \node[fill=none,draw=none] (i2) at (4,0.5) {$i_2=p_{k}$};
              \node[fill=none,draw=none] (j1) at (0,-1.5) {$j_1$};
              \node[fill=none,draw=none] (j2) at (4,-1.5) {$j_2$};
              \node[fill=none,draw=none] (p1) at (1,0.5) {$p_1$};
              \node[fill=none,draw=none] (pk1) at (3,0.5) {$p_{k-1}$};
              \node[fill=none,draw=none] (e1) at (0.7,-0.7) {$e_1$};
              \node[fill=none,draw=none] (e2) at (3.3,-0.7) {$e_2$};
              \draw[dotted,thick] (B)--(C);
              \draw[dotted,thick] (B)--(E);
              \draw[dotted,thick] (C)--(F);
              \draw[] (B)--(A)--(E);
              \draw[] (C)--(D)--(F);

            \end{tikzpicture}
            \caption{The subgraph induced by (g) and a shortest path between its two edges.}\label{fig:subgraph2cellspath}
          \end{figure}

          Assume that $V_i$ and $V_j$ induce the empty subgraph. Since $G$ is connected, there must be some part $V_k\neq V_j$ such that $G[V_i\cup V_k]$ is not empty, so we may instead consider one of the other subgraphs to determine $\nu$ on $V_i$. If $V_i$ and $V_j$ induce subgraph (g), consider a shortest path between either endpoint of edge $\{i_1,j_1\}$ and $\{i_2,j_2\}$. Figure \ref{fig:subgraph2cellspath} shows the general shape of the graph induced by this path and $V_i$, $V_j$, assuming without loss of generality that $i_1$ and $i_2$ are the endpoints of the path. Note that there always exists an induced path of length three, unless both edges $e_1$ and $e_2$ exist and $p_1 = p_{k-1}$. Since $G$ is a cograph, we only need to consider the latter case. Let $v$ be the vertex which forms a cell with $p_1$. This vertex must have the same number of neighbors in $V_i$, so it is adjacent to $i_1$ and $i_2$. Then $V_i$ and $\{p_1,v\}$ induce either case (b) or (d), hence $\nu$ is constant over~$V_i$.
	\end{proof}
	
	As equitable partitions are always weight-equitable, we obtain the following corollary.
	
	\begin{corollary}
          \label{cor:wrCographs}
          Let~$G$ be a connected cograph.
          Then, $G$ has an automorphism being an involution without fixed points if and only if~$G$ admits a (weight-) equitable partition with~$\frac{n}{2}$ cells each having size~2.
	\end{corollary}
	\begin{proof}
          The result follows directly from Lemma \ref{lem:automorphismIFFpartitionRegular} and Lemma \ref{lem:automorphismIFFpartitionWeightRegular}.
	\end{proof}
	Note that partitions with cells of size two are not necessarily the
        finest equitable partitions of cographs. Figure \ref{fig:notFinest}
        shows an example of a cograph and an equitable 2-homogeneous partition which is not the finest.
	In fact, one can verify that any refinement of this partition is
        also equitable.
        Moreover, Lemma~\ref{lem:automorphismIFFpartitionWeightRegular}
        cannot be generalized to $c$-homogeneous partitions with~$c > 2$ as
        the cograph in Figure~\ref{fig:kbalancedWRnotEquitable} admits a
        weight-equitable partition which is not equitable.
        \begin{figure}[t]
          \centering
          \begin{subfigure}[b]{0.45\linewidth}
            \centering
            \begin{tikzpicture}
              \node[shape=circle,fill=black,inner sep=2pt] (A) at (0,0) {};
              \node[shape=circle,fill=black,inner sep=2pt] (B) at (1,0) {};
              \node[shape=rectangle,fill=black,inner sep=2.4pt] (C) at (2,1) {};
              \node[shape=rectangle,fill=black,inner sep=2.4pt] (D) at (-1,1) {};
              \node[regular polygon,regular polygon sides=3,fill=black,inner sep=1.4pt] (E) at (-1,-1) {};
              \node[regular polygon,regular polygon sides=3,fill=black,inner sep=1.4pt] (F) at (2,-1) {};
              \draw[] (D)--(A)--(F)--(B)--(D);
              \draw[] (C)--(A)--(E)--(B)--(C);
            \end{tikzpicture}
            \caption{A cograph with an equitable partition of 2-cells. Any refinement of its partition is again equitable.}\label{fig:notFinest}
          \end{subfigure}
          \hfill
          \centering
          \begin{subfigure}[b]{0.45\linewidth}
            \centering
            \begin{tikzpicture}
              \node[shape=circle,fill=black,inner sep=2pt] (A) at (0,0) {};
              \node[shape=circle,fill=black,inner sep=2pt] (B) at (1,0) {};
              \node[shape=circle,fill=black,inner sep=2pt] (C) at (2,-1) {};
              \node[shape=rectangle,fill=black,inner sep=2.4pt] (D) at (0,-1) {};
              \node[shape=rectangle,fill=black,inner sep=2.4pt] (E) at (1,-1) {};
              \node[shape=rectangle,fill=black,inner sep=2.4pt] (F) at (2,0) {};
              \node (G) at (1,-1.5) {};
              \draw[] (A)--(B)--(F)--(C)--(E)--(D)--(A);
              \draw[] (A)--(E)--(B);
              \draw[] (D)--(B)--(C);
              \draw[] (E)--(F);
            \end{tikzpicture}
            \caption{A cograph with a $3$-balanced weight-equitable partition, which is not equitable.}\label{fig:kbalancedWRnotEquitable}
          \end{subfigure}
          \caption{Examples of cographs.}
          \label{fig:cographs}
        \end{figure}
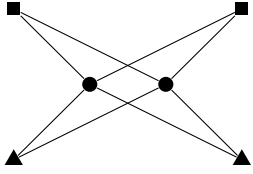
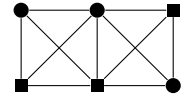
        We now turn the focus back to finding 2-homogeneous weight-equitable
        partitions of cographs.
	To devise an algorithm finding such partitions, we make use of an
        alternative characterization of cographs.
        \begin{proposition}[Corneil et al.~\cite{CLHB1981}]
          A cograph is defined recursively using the following three rules:
          \begin{enumerate}[(i)]
          \item A graph on a single vertex is a cograph.
          \item If $G_1,\dots,G_k$ are cographs, then so is $G_1\cup \dots \cup G_k$.
          \item If $G$ is a cograph, then so is its complement $\overline{G}$.
          \end{enumerate}
        \end{proposition}
	Note that we may equivalently replace (iii) by the condition that
        the join of two cographs is again a cograph.
        Using this characterization, the structure of a cograph can uniquely be represented by a rooted tree.
	
	Let $G$ be a cograph.
        A \emph{cotree} of $G$ is a rooted tree whose inner vertices each
        have a label 0 or 1.
        A leaf vertex corresponds to an induced subgraph on a single vertex
        and the subtree rooted at a vertex with label 0 or 1 corresponds to
        the union or join respectively of the subgraphs represented by its
        children.
        Note that two vertices form an edge in $G$ if and only if their
        least common ancestor in the cotree has label 1.
        If we require the labels on a root-leaf path to be alternating,
        this tree is unique, see Corneil et al.~\cite{CLHB1981}.
        The same authors observed that the graph isomorphism problem is therefore polynomial-time solvable for cographs.

	To show that finding 2-homogeneous partitions is easy for cographs,
        we bound the complexity of the problem of finding a
        fixed-point-free involutionary automorphism in a cograph.
        To do this, we need a preliminary lemma.
        Here, $\psi|_V$ denotes the restriction of a map $\psi$ to the set
        $V$ and $G[S]$ the subgraph of the graph $G$ induced by a subset
        $S$ of the vertices.
	\begin{lemma}
          \label{lem:automorphismCotree}
          Let~$G$ be a cograph with unique cotree $T$ with
          alternating 0/1-labels.
          Then, $\phi\colon V \rightarrow V$ is an automorphism of $G$ if
          and only if there exists an automorphism $\psi$ on $T$ such that
          $\psi|_V = \phi$ and which respects the 0/1-labeling.
	\end{lemma}
	\begin{proof}
          We prove the first implication that relates automorphisms~$\phi$
          of~$G$ with automorphisms~$\psi$ of~$T$ by induction.
          If $T$ has depth 0 or 1, the statement is trivial.
          Suppose that the claim holds for cographs with cotrees of depth
          at most $d$. Let $G$ be a cograph with a cotree $T$ of depth
          $d+1$ with root $r$ and let $\phi$ be an automorphism of $G$.
          Note that the cotree $T'$ of $\phi(G)$ is the same as $T$, up to
          a renaming of the leaves.
          We will show that there exists a label-preserving automorphism
          $\psi$ such that $\psi(T) = T'$ by considering a case distinction
          on the root label.

          If $r$ has label 0, the cograph is disconnected, but each subtree
          rooted at a child of $r$ induces a connected cograph.
          Since $\phi$ is an automorphism, it maps isomorphic connected
          components of $G$ onto each other.
          If a connected component $C$ is mapped onto a different component
          $C'$, let $\psi$ map the subtree of $C$ to the subtree of $C'$.
          Note that, since the components are isomorphic, their cotrees are
          isomorphic, which means that there exists such a mapping which
          preserves the 0/1-labels and has the property~$\psi|_C =
          \phi|_C$.
          If $\phi$ maps a component $C$ onto itself, then by the induction
          hypothesis, there exists an automorphism $\psi'$ of $T[C]$ which
          respects the 0/1-labels such that $\psi'|_C = \phi|_C$.
          Let~$\psi$ take the values of $\psi'$ when restricted to the
          cotree of $C$.
          Finally set $\psi(r) = r$.
          Then $\psi$ is an automorphism of $T$ which satisfies the statement.

          Assume that $r$ has label 1.
          In this case, the cograph is connected, but each subtree rooted at a child of $r$ corresponds to a disconnected cograph.
          Let $u$ and $v$ be vertices such that $\phi(u) = v$ and let~$T^u$ and $T^v$ be the subtrees of $r$ containing them.
          Denote the cograph induced by the tree $T^u$ by~$G[T^u]$.
          Each vertex in $G[T^u]$ which is not a neighbor of $u$ must also be mapped to a leaf of $T^v$.
          If not, its least common ancestor with $v$ is the root, which has label 1.
          This means that a nonneighbor of $u$ is mapped to a neighbor of $v$, contradicting the fact that $\phi$ is an automorphism.
          Therefore, all connected components of $G[T^u]$, apart from possibly the one containing $u$, must be mapped into $G[T^v]$.
          Since $G[T^u]$ is disconnected, there exists a leaf $w\in T^u$ which, as a vertex of $G[T^u]$, is not in the same connected component as $u$.
          By the same argument, each vertex in the connected component of $u$ must also be mapped into $G[T^v]$, else its least common ancestor with $w$ will be the root with label 1.
          Hence each cograph given by a subtree of $r$ is mapped entirely into a cograph induced by another subtree of $r$.
          As a result, it is also not possible that $\phi$ maps the cographs of two distinct subtrees into the cograph of the same larger subtree, as this would force a larger cograph to be split among smaller ones.
          This means that we can see $\phi$ as a permutation of the cographs induced by each subtree rooted at a child of $r$.
          Subgraphs that are mapped onto each other must be isomorphic, hence there exists an isomorphism between their cotrees which coincides with $\phi$ when restricted to the leaves.
          Let $\psi$ take the values of these isomorphisms.
          If a subgraph $G'$ is mapped onto itself by $\phi$, the induction hypothesis gives us a label-preserving automorphism $\psi_{G'}$ of its cotree and we can set $\psi|_{G'} = \psi_{G'}$.
          Setting $\psi(r) = r$ completes the required automorphism of $T$.

          Conversely, suppose that $\psi$ is an automorphism of $T$ which respects the 0/1-labeling and let $\phi \define \psi|_V$ be the restriction of $\psi$ to $V$.
          By the definition of an automorphism, $(\psi(u),\psi(v))\in E(T)$ if and only if $\{u,v\}\in E(T)$.
          In particular, this means that the least common ancestor  of every pair of leaf vertices is preserved under $\psi$.
          Then $\phi$ maps edges of $G$ to edges and nonedges to nonedges, so it is an automorphism of $G$.
	\end{proof}
	
	Lemma \ref{lem:automorphismCotree} allows us to test automorphism on cographs by considering the cotree.
        In general, the automorphism problem on rooted labeled trees is
        known to be polynomial-time solvable, see Colbourn and Booth~\cite{CK1981}, hence it is polynomial for cographs.
        However, in the context of Corollary \ref{cor:wrCographs} we need to determine the existence of a particular type of automorphism which is also an involution without fixed points.
        In terms of the cotree, this means that there should be an automorphism which swaps the leaves pairwisely without leaving any in place.
        Algorithm \ref{alg:treeAuto} proposes a recursive procedure which determines the existence of such a mapping for general rooted trees.
        Here $T_v$ denotes the subtree of a tree $T$ rooted at vertex $v$.
        The algorithm assumes that the input tree has been labeled using
        the $j$-numbering procedure by Colbourn and Booth~\cite{CK1981}.
        These numbers are assigned in a top-down fashion to each vertex of the tree and, together with the depth of a vertex, partition the tree into its orbits under the automorphism group.
        A $j$-numbering can be computed in linear time, hence the running time of Algorithm \ref{alg:treeAuto} is polynomial.

        Before we show correctness of Algorithm~\ref{alg:treeAuto}, we provide some intuition.
        Consider the rooted tree $T$ given in Figure \ref{fig:algorithmexample}.
        It is clear that an automorphism of $T$ could swap the subtrees rooted at $u$ and~$v$, as they are isomorphic and have the same parent.
        This is a fixed-point-free involution on the leaves that descend from $u$ and $v$.
        The subtree of $w$ cannot be mapped to another part of the graph in its entirety, but if we go one level down, we see that interchanging the children of $w$ also results in a fixed-point-free involution on the remaining leaves.

        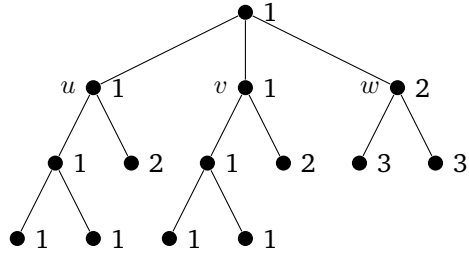
\begin{figure}[t]
            \centering
			\begin{tikzpicture}
				\node[shape=circle,fill=black,inner sep=2pt,label=right:1] (a) at (0,0) {};
				\node[shape=circle,fill=black,inner sep=2pt,label=right:1] (b) at (1,0) {};
				\node[shape=circle,fill=black,inner sep=2pt,label=right:1] (c) at (0.5,1) {};
				\node[shape=circle,fill=black,inner sep=2pt,label=right:2] (d) at (1.5,1) {};
				\node[shape=circle,fill=black,inner sep=2pt,label=left:$u$,label=right:1] (e) at (1,2) {};
				\node[shape=circle,fill=black,inner sep=2pt,label=right:1] (f) at (2.5,1) {};
				\node[shape=circle,fill=black,inner sep=2pt,label=right:2] (g) at (3.5,1) {};
				\node[shape=circle,fill=black,inner sep=2pt,label=left:$v$,label=right:1] (h) at (3,2) {};
                \node[shape=circle,fill=black,inner sep=2pt,label=right:1] (i) at (2,0) {};
				\node[shape=circle,fill=black,inner sep=2pt,label=right:1] (j) at (3,0) {};
                \node[shape=circle,fill=black,inner sep=2pt,label=right:1] (k) at (3,3) {};
                \node[shape=circle,fill=black,inner sep=2pt,label=left:$w$,label=right:2] (l) at (5,2) {};
                \node[shape=circle,fill=black,inner sep=2pt,label=right:3] (m) at (4.5,1) {};
                \node[shape=circle,fill=black,inner sep=2pt,label=right:3] (n) at (5.5,1) {};
                \draw[] (a) -- (c) -- (b);
                \draw[] (i) -- (f) -- (j);
                \draw[] (c) -- (e) -- (d);
                \draw[] (f) -- (h) -- (g);
                \draw[] (e) -- (k) -- (h);
                \draw[] (m) -- (l) -- (n);
                \draw[] (k) -- (l);
			\end{tikzpicture}
			\caption{A $j$-numbered tree which admits a fixed-point-free involution on the leaves.}\label{fig:algorithmexample}
        \end{figure}

        This idea can be formalized as follows.
        Start at the root of the tree and consider all of its children with a particular $j$-number.
        Since they are at the same depth, these vertices share an orbit under $\text{Aut}(T)$.
        If there is an even number of them, their subtrees can be interchanged pairwisely, as illustrated in the example above.
        Otherwise, one subtree cannot be paired and the procedure is repeated on the root of this subtree.
        The algorithm ends when either all leaves have been interchanged or when a vertex is encountered which has an odd number of children that are leaves.
        In the latter case, we can neither recurse nor interchange in pairs, hence $T$ does not have the required automorphism.

	\begin{algorithm2e}[t]
          \SetAlgoLined
          \SetKwFunction{FhasNiceAutomorphism}{hasNiceAutomorphism}
          \SetKwProg{Fn}{}{}{}
          \Input{A (labeled) rooted tree $T$ with root $r$ and $j$-numbers
            assigned to each vertex, following the procedure from \cite[Lemma 2.1]{CK1981}}
          \Output{Does $T$ admit an automorphism which is a fixed-point-free involution on the leaves?}

          \Fn{\FhasNiceAutomorphism{$T$, $r$}}{
            \For{each child $v$ of $r$ with distinct $j$-number}{
              let $k_v$ be the number of children with the same $j$-number as $v$\\
              \uIf {$k_v$ is odd} {
                \uIf {$v$ is a leaf} {
                  \Return false\\
                }
                \uElse {
                  \uIf {\FhasNiceAutomorphism{$T_v$, $v$} = false}{
                    \Return false\\
                  }
                }
              }
            }
            \Return true\\
          }
          \caption{\FuncSty{hasNiceAutomorphism}}
          \label{alg:treeAuto}
	\end{algorithm2e}
	
	\begin{lemma}
          \label{lem:fixedPointFreeAutomosphismCotree}
          Let~$T$ be a (labeled) rooted tree with at least two vertices.
          Algorithm \ref{alg:treeAuto} returns ``true'' if and only if $T$ admits an automorphism which is a fixed-point-free involution on the leaves.
          The running time of the algorithm is $O(n^2)$.
	\end{lemma}
	\begin{proof}
          We will prove the claim by induction.
          As a base case, suppose that $T$ has depth one, i.e., the
          corresponding cograph is either a clique or does not contain any edge.
          Then every child of the root has the same $j$-number, and all
          nodes can be exchanged arbitrarily by an automorphism.
          If there is an even number of them, the algorithm will return ``true'', as the first if-condition is never fulfilled.
          We can pair up these vertices and exchange them to obtain an involutionary automorphism without fixed leaves.
          If the root has an odd number of children, they cannot be partitioned into pairs and such an automorphism does not exist.
          Algorithm \ref{alg:treeAuto} will correctly return ``false'', as the first if-statement is triggered.

          Assume that the statement holds for rooted trees of depth $d$ and let $T$ be a tree of depth $d+1$ with root $r$.
          Consider a $j$-number among the children of $r$ and let $v_1,\dots,v_k$ be the children which share this number.
          The subtrees rooted at $v_1,\dots,v_k$ are isomorphic, hence if $k$ is even, we can pair them up arbitrarily and map the pairs onto each other.

          If all $j$-numbers appear an even number of times, the union of these partial mappings constitutes an automorphism of $T$ which is a fixed-point-free involution on the leaves.
          Algorithm \ref{alg:treeAuto} will return ``true'', as no $j$-number is shared by an odd number of children.

          Suppose instead that $k$ is odd.
          In this case, it is not possible to pair $v_1,\dots,v_k$ and permute their subtrees pairwisely.
          Note that it is also not allowed to permute subtrees in larger cycles, as this no longer creates an involution on the leaves.
          Without loss of generality, we may pair up and exchange the subtrees of all but $v_k$.
          If $T$ has an automorphism as required, then $T_{v_k}$ must have an automorphism which is a fixed-point-free involution on its leaves and vice versa.
          By the induction hypothesis, this is the case if and only if the algorithm returns ``true'' for this subtree.
          If ``false'' is returned, we know that such an automorphism does not exist, hence we also cannot find one for $T$.
          The algorithm halts and returns ``false'' for this instance.
          If for each odd $j$-class the answer is ``true'', then the partial mappings of the classes can be merged into a valid automorphism as in the even case, and algorithm correctly returns ``true''.

          Finally, we bound the running time of Algorithm \ref{alg:treeAuto}.
          In the worst case, all children of each vertex have distinct $j$-numbers, except for the leaves, whose $j$-numbers all come in even numbers.
          Then, the algorithm recurses on all vertices of $T$ which are not a leaf.
          For each recursion step, it partitions the children of a given vertex by their $j$-numbers.
          This can be done in linear time and a vertex certainly has no more than $n$ children.
          Hence Algorithm \ref{alg:treeAuto} has a worst case running time of~$O(n^2)$.
	\end{proof}
	
	Since the number of vertices of a tree equals at most twice the
        number of leaves, the running time of Algorithm \ref{alg:treeAuto}
        is also quadratic in the number of vertices of the cograph itself.
        Combining the previous results from this section, we obtain our last main result.
	
	\begin{theorem}
          \label{th:mainresultcographs}
          Let $G$ be a cograph. The problem of deciding whether $G$
          admits a (weight-) equitable partition with $\frac{n}{2}$ cells of size 2 can be solved in $O(n^2)$ time.
	\end{theorem}
	
	Note that Algorithm \ref{alg:treeAuto} can easily be modified to keep track of the partial mappings and return an automorphism $\psi$ of the cotree which is a fixed-point-free involution on the leaves.
        Restricting $\psi$ to the leaves gives a fixed-point-free
        involutionary automorphism of $G$, whose orbits form a (weight-) equitable partition of $G$.
        Hence computing (weight-) equitable partitions with $\frac{n}{2}$
        cells of size 2 can also be done in quadratic time.
        We conclude with some remarks for finding general~$c$-homogeneous
        partitions of cographs.
	
	Algorithm \ref{alg:treeAuto} can be adapted to a more general setting.
        In its current form, it determines the existence of a fixed-point-free involution~$\gamma$ on a given cograph.
        Phrased differently, if we consider the group generated
        by~$\gamma$, then it acts as $\text{Sym}_2$, the symmetric group of order~2, on each cell of the
        (weight-) equitable partition.
        To find~$c$-homogeneous partitions, we can generalize the idea of
        Algorithm~\ref{alg:treeAuto} by partitioning the subtrees according
        to their~$j$-number into groups of size~$c$.
        By recursing as usual when this is not the case, we obtain an
        algorithm which determines whether $G$ admits a $c$-homogeneous
        partition such that there exists a subgroup of~$G$'s automorphism
        group~$\text{Aut}(G)$ that acts as the symmetric group $\text{Sym}_c$ on each cell.

	To exploit the generality of the algorithm, a link between automorphisms and $c$-homogeneous equitable partitions is required, similar to Lemma \ref{lem:automorphismIFFpartitionRegular}.
        Let $G$ be a graph with a $c$-homogeneous partition.
        It is easy to see that if a subgroup of $\text{Aut}(G)$ acts as $\text{Sym}_c$ on its cells, then the partition must be equitable.
        However, the converse need not hold, as illustrated by Example \ref{ex:WRnotSym}.
        This means that when our algorithm is successful, an equitable partition is guaranteed to exist, but it may produce false negatives.
	
	\begin{example}\label{ex:WRnotSym}
          Consider a 4-cycle.
          The trivial partition of this graph is equitable, because the graph is regular.
          However, its automorphism group is the dihedral group, which is a
          proper subgroup of~$\text{Sym}_4$.
	\end{example}
	
\subsection{Joins of weight-equitable partitions}\label{sec:experiments}

In Section \ref{sec:characterizations}, it was shown that the join of two weight-equitable partitions is again weight-equitable.
If we have a number of weight-equitable partitions at hand, this gives us a method to construct coarser ones.
An interesting question is, on the one hand, how close one can get to the coarsest nontrivial partition with a small number of join operations.
On the other hand, one may wonder how many joins can be done before obtaining the trivial partition.
In this section, we study these questions empirically for cographs, using Algorithm \ref{alg:treeAuto} to construct the initial weight-equitable partitions.
Recall that for cographs, weight-equitable partitions and equitable partitions coincide, hence Bastert's algorithm~\cite{B1999} finds the coarsest (weight-) equitable partition in polynomial time.
However, for general graphs no algorithm is known and the join operation may provide a useful approximation method.

The setup is the following.
For a small even integer $n$, generate all connected cographs on $n$ vertices.
For each cograph $G$, use Algorithm \ref{alg:treeAuto} to determine whether it admits a 2-homogeneous partition.
If not, it is discarded, and if it does, a 2-homogeneous partition is found.
As $n$ is small, the generators $g_1\dots, g_k$ of $\text{Aut}(G)$ can be computed quickly.
Nine additional 2-homogeneous partitions are sampled by generating $m_i\in [\text{ord}(g_i)]$ and applying the automorphism $\prod_{i=1}^k g_i^{m_i}$ to the given partition.
We take the join of every possible subset of the ten 2-homogeneous partitions and count the number of cells of the resulting weight-equitable partitions.

Figure \ref{fig:joinsmallgraphs} summarizes the results for $n=4,6,\dots,14$.
The size of the circle at position $(x,y)$ represents how often joining $x$ partitions results in a partition with $y$ cells.
Every column is normalized by the number of ways to choose $x$ partitions.
Note that for $n=4$ and $n=6$ the last columns are relatively sparse, because many cographs of this size do not admit ten distinct 2-homogeneous partitions.
For the larger graphs, it is most likely to obtain three to five cells by merging ten partitions.
In very few cases it results in the trivial partition.

\begin{figure}[t]
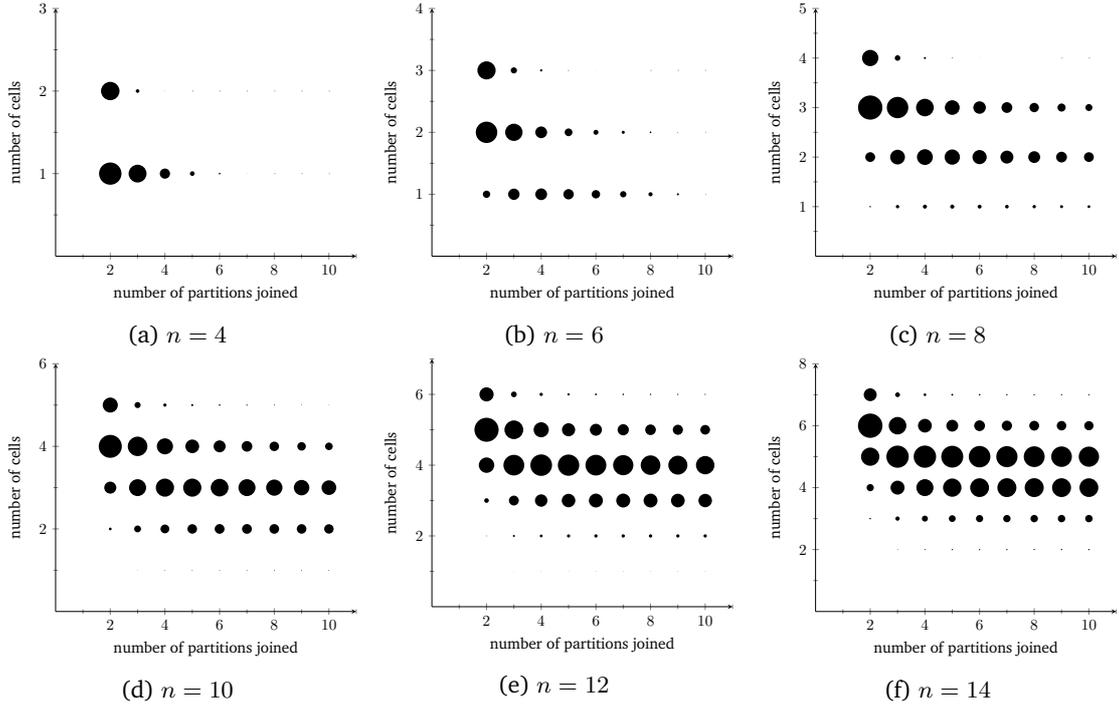

    \centering
    \begin{minipage}{.33\textwidth}
        \centering
        \subcaptionbox{$n=4$}{
        \begin{adjustbox}{width=0.95\textwidth}
        \makeplot{\figa}{11}{3}{10}
        \end{adjustbox}
        }
    \end{minipage}%
    \begin{minipage}{0.33\textwidth}
        \centering
        \subcaptionbox{$n=6$}{
        \begin{adjustbox}{width=0.95\textwidth}
        \makeplot{\figb}{11}{4}{1.6}
        \end{adjustbox}
        }
    \end{minipage}
    \begin{minipage}{0.33\textwidth}
        \centering
        \subcaptionbox{$n=8$}{
        \begin{adjustbox}{width=0.95\textwidth}
        \makeplot{\figc}{11}{5}{0.32}
        \end{adjustbox}
        }
    \end{minipage}
    \begin{minipage}{.33\textwidth}
        \centering
        \subcaptionbox{$n=10$}{
        \begin{adjustbox}{width=0.95\textwidth}
        \makeplot{\figd}{11}{6}{0.06}
        \end{adjustbox}
        }
    \end{minipage}%
    \begin{minipage}{0.33\textwidth}
        \centering
        \subcaptionbox{$n=12$}{
        \begin{adjustbox}{width=0.95\textwidth}
        \makeplot{\fige}{11}{7}{0.013}
        \end{adjustbox}
        }
    \end{minipage}
    \begin{minipage}{0.33\textwidth}
        \centering
        \subcaptionbox{$n=14$}{
        \begin{adjustbox}{width=0.95\textwidth}
        \makeplot{\figf}{11}{8}{0.0027}
        \end{adjustbox}
        }
    \end{minipage}
    \caption{Distribution of coarseness with respect to the number of join operations for small graphs.}
    \label{fig:joinsmallgraphs}
\end{figure}

To see whether this patterns continues for larger graphs, we repeat the procedure for 200 cographs on $20,\ 30,\ 40$ and $50$ vertices which admit a 2-homogeneous weight-equitable partition.
In this case, merging ten partitions most likely gives a partition with four cells.
However, four partitions is already enough to obtain a similar
distribution.

\begin{figure}[t]
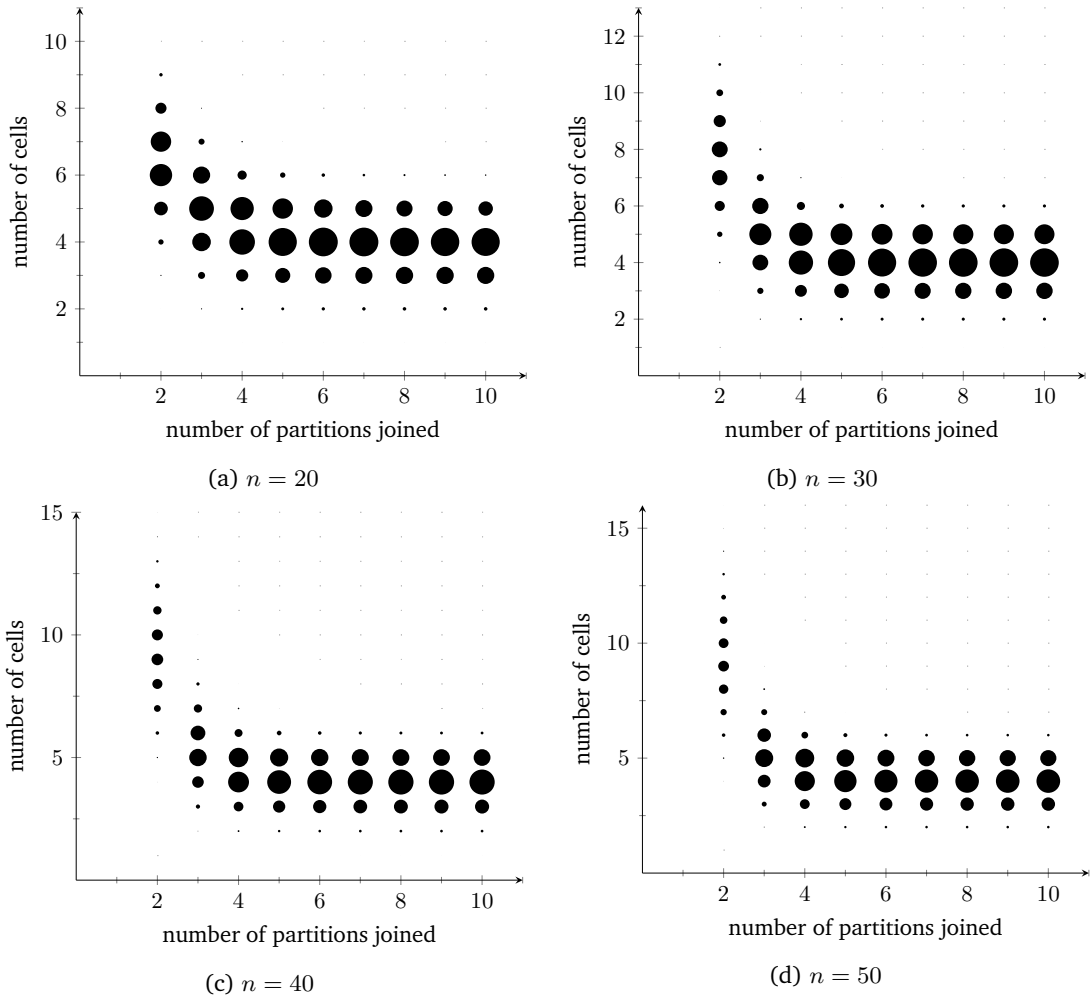

    \centering
    \begin{minipage}{.49\textwidth}
        \centering
        \subcaptionbox{$n=20$}{
        \begin{adjustbox}{width=0.95\textwidth}
        \makeplot{\largefiga}{11}{11}{0.025}
        \end{adjustbox}
        }
    \end{minipage}%
    \begin{minipage}{.49\textwidth}
        \centering
        \subcaptionbox{$n=30$}{
        \begin{adjustbox}{width=0.95\textwidth}
        \makeplot{\largefigb}{11}{13}{0.022}
        \end{adjustbox}
        }
    \end{minipage}
    \begin{minipage}{0.49\textwidth}
        \centering
        \subcaptionbox{$n=40$}{
        \begin{adjustbox}{width=0.95\textwidth}
        \makeplot{\largefigc}{11}{15}{0.019}
        \end{adjustbox}
        }
    \end{minipage}
    \begin{minipage}{0.49\textwidth}
        \centering
        \subcaptionbox{$n=50$}{
        \begin{adjustbox}{width=0.95\textwidth}
        \makeplot{\largefigd}{11}{16}{0.018}
        \end{adjustbox}
        }
    \end{minipage}

    \caption{Distribution of coarseness with respect to the number of join operations.}
    \label{fig:joinlargegraphs}
\end{figure}

We conclude that using the join operation discussed in
Corollary~\ref{lemma:changodsil2} for just a few times is able to generate very
coarse partitions even if we start with (weight-) equitable partitions that
are very fine.
Thus, although we do not know how to compute a coarsest non-trivial
weight-equitable partition, the join operator allows us to get a reasonably
good approximation.

	\section*{Acknowledgments}
	The research of A. Abiad is partially supported by the FWO grant 1285921N.
    We are also grateful to Frits Spieksma for his comments on the article.

	\bibliographystyle{abbrv}
    \bibliography{bibliography}

\begin{thebibliography}{10}

\bibitem{mscthesisaida}
A.~Abiad.
\newblock Some applications of linear algebra in spectral graph theory.
\newblock Master's thesis, Polytechnic University of Catalonia, 2011.

\bibitem{A2019}
A.~Abiad.
\newblock A characterization and an application of weight-regular partitions of
  graphs.
\newblock {\em Linear Algebra and its Applications}, 569:162--174, 2019.

\bibitem{acfns2020}
A.~Abiad, G.~Coutinho, M.~A. Fiol, B.~D. Nogueira, and S.~Zeijlemaker.
\newblock Optimization of eigenvalue bounds for the independence and chromatic
  number of graph powers, 2020.
\newblock arXiv:2010.12649.

\bibitem{B1999}
O.~Bastert.
\newblock Computing equitable partitions of graphs.
\newblock {\em Match}, 40:265--272, 1999.

\bibitem{BS1979}
A.~E. Brouwer and A.~Schrijver.
\newblock Uniform hypergraphs.
\newblock In {\em Packing and covering in combinatorics (Study week {"}Stapelen
  en overdekken{"}, Amsterdam, The Netherlands, June 5-9, 1978)}, volume 108 of
  {\em Mathematical Centre Tracts}, pages 39--73. Stichting Mathematisch
  Centrum, 1979.

\bibitem{CG1997}
A.~Chan and C.~D. Godsil.
\newblock Symmetry and eigenvectors.
\newblock In {\em Graph symmetry}, pages 75--106. Springer, 1997.

\bibitem{CK1981}
C.~J. Colbourn and K.~S. Booth.
\newblock Linear time automorphism algorithms for trees, interval graphs, and
  planar graphs.
\newblock {\em SIAM Journal on Computing}, 10(1):203--225, 1981.

\bibitem{CC}
D.~G. Corneil and C.~C. Gotlieb.
\newblock An efficient algorithm for graph isomorphism.
\newblock {\em Journal of the ACM}, 17(1):51--64, 1970.

\bibitem{CLHB1981}
D.~G. Corneil, H.~Lerchs, and L.~S. Burlingham.
\newblock Complement reducible graphs.
\newblock {\em Discrete Applied Mathematics}, 3(3):163--174, 1981.

\bibitem{F1999}
M.~A. Fiol.
\newblock Eigenvalue interlacing and weight parameters of graphs.
\newblock {\em Linear algebra and its applications}, 290(1):275--301, 1999.

\bibitem{F2001}
M.~A. Fiol.
\newblock On pseudo-distance-regularity.
\newblock {\em Linear Algebra and its Applications}, 323(1):145--165, 2001.

\bibitem{FG1999}
M.~A. Fiol and E.~Garriga.
\newblock On the algebraic theory of pseudo-distance-regularity around a set.
\newblock {\em Linear Algebra and its Applications}, 298(1):115--141, 1999.

\bibitem{G2019}
E.~Ghorbani.
\newblock Cographs: Eigenvalues and {Dilworth} number.
\newblock {\em Discrete Mathematics}, 342(10):2797--2803, 2019.

\bibitem{G2018}
E.~Ghorbani.
\newblock Spectral properties of cographs and p5-free graphs.
\newblock {\em Linear and Multilinear Algebra}, 67(8):1701--1710, 2019.

\bibitem{Cd40}
C.~D. Godsil.
\newblock {\em Algebraic combinatorics}, volume~6.
\newblock CRC Press, 1993.

\bibitem{G1997}
C.~D. Godsil.
\newblock Compact graphs and equitable partitions.
\newblock {\em Linear Algebra and its Applications}, 255(1):259--266, 1997.

\bibitem{GR2001}
C.~D. Godsil and G.~F. Royle.
\newblock {\em Algebraic graph theory}, volume 207.
\newblock Springer, 2001.

\bibitem{GKMS2014}
M.~Grohe, K.~Kersting, M.~Mladenov, and E.~Selman.
\newblock Dimension reduction via colour refinement.
\newblock In A.~S. Schulz and D.~Wagner, editors, {\em Algorithms - ESA 2014},
  pages 505--516. Springer, 2014.

\bibitem{H1970}
A.~J. Hoffman.
\newblock On eigenvalues and colorings of graphs.
\newblock In {\em Selected Papers of Alan J Hoffman: With Commentary}, pages
  407--419. World Scientific, 2003.

\bibitem{JTT2018}
D.~Jacobs, V.~Trevisan, and F.~Colman~Tura.
\newblock Eigenvalue location in cographs.
\newblock {\em Discrete Applied Mathematics}, 245:220--235, 2018.

\bibitem{J}
H.~Jung.
\newblock On a class of posets and the corresponding comparability graphs.
\newblock {\em Journal of Combinatorial Theory, Series B}, 24(2):125--133,
  1978.

\bibitem{Lubiw1981}
A.~Lubiw.
\newblock Some {N}{P}-complete problems similar to graph isomorphism.
\newblock {\em SIAM Journal on Computing}, 10(1):11--21, 1981.

\bibitem{MT2016}
A.~Mohammadian and V.~Trevisan.
\newblock Some spectral properties of cographs.
\newblock {\em Discrete Mathematics}, 339(4):1261--1264, 2016.

\end{thebibliography}
	
\end{document}